\documentclass[11pt,a4paper,twoside]{article}
\usepackage{geometry} 
\usepackage{bbm,mathrsfs,eurosym} 
\usepackage{amsmath, amsthm, amsfonts} 
\usepackage{color}

\newtheorem{theorem}{Theorem}
\newtheorem{corollary}{Corollary}
\newtheorem{lemma}{Lemma}

\theoremstyle{remark}
\newtheorem{remark}{Remark}

\theoremstyle{remark}

\theoremstyle{definition}
\newtheorem{definition}{Definition}

\title{Existence and stability of measure solutions for BSDE with generators of quadratic growth}
\author{Alexander Fromm, Peter Imkeller, Jianing Zhang \vspace{3mm} \\ Institut f\"ur Mathematik\\
Humboldt-Universit\"at zu Berlin\\
Unter den Linden 6\\
10099 Berlin\\ Germany}
\date{April 28, 2011}

\begin{document}

\maketitle

\begin{abstract}
\noindent With an emphasis on generators with quadratic growth in the
control variable we consider measure solutions of BSDE, a solution concept
corresponding to the notion of risk neutral measure in mathematical
finance. In terms of measure solutions, solving a BSDE reduces to
martingale representation with respect to an underlying filtration. Measure
solutions related to measures equivalent to the historical one provide
classical solutions. We derive the existence of measure solutions in
scenarios in which the generating functions are just continuous, of at most
linear growth in the control variable (corresponding to generators of at
most quadratic growth in the usual sense), and with a random bound in the
time parameter whose stochastic integral is a BMO martingale. Our main
tools include a stability property of sequences of measure solutions, for
which a limiting solution is obtained by means of the weak convergence of
measures.
\end{abstract}\vspace{2mm}

\section*{Introduction}

The most efficient formulation of pricing and hedging contingent claims on
complete financial markets is given by the elegant notion of risk neutral
or martingale measures. From its perspective, pricing amounts to taking
expectations, while hedging boils down to pure conditioning and using
martingale representation.\par\smallskip

From the perspective of stochastic control theory, hedging consists in
choosing appropriate strategies to steer a portfolio into a terminal random
endowment the portfolio holder has to ensure. Backward stochastic
differential equations (BSDE) are tailor-made for this purpose. On a
Brownian basis, a BSDE with terminal variable $\xi$ at time horizon $T$ and
generator $f$ is solved by a pair of processes $(Y,Z)$ on the interval
$[0,T]$ satisfying \begin{equation}\label{meassol-eq:introduction} Y_t = \xi -
\int_t^T Z_s \mathrm{d}W_s + \int_t^T f(s, Y_s, Z_s) \mathrm{d}s, \quad
t\in[0,T].\end{equation} In case $f=0$, the solution just requires an
application of the martingale representation theorem, and $Z$ will be given
as the stochastic integrand therein. Generators $f$ in BSDE derived from
many problems of utility maximization or risk minimization turn out to be
quadratic in the control variable $z$, and have been treated in a large
number of papers starting with the pioneering one by Kobylanski
\cite{meassol-koby}.
\par\medskip

In \cite{meassol-meas} , the notion of \emph{measure solutions} for BSDE has been
introduced with the aim to extend the passage from the historical to the
risk neutral world to a more general framework. In analogy with martingale
measures in hedging which eliminate drifts in the underlying market
dynamics, these solutions of BSDE are given by probability measures under
which their generators are seen as vanishing. Determining a measure
$\mathbb{Q}$ under which the generator vanishes amounts to performing a
Girsanov change of probability that eliminates it. We therefore have to
look at the BSDE in the form
\begin{equation}\label{meassol-eq:introduction2} Y_t = \xi - \int_t^T Z_s \left[ \mathrm{d}W_s -
g(s, Y_s, Z_s) \mathrm{d}s \right],\quad t\in[0,T],\end{equation} where
$\displaystyle z\cdot g(s,y,z) = f(s,y,z)$, and study the measure
$$\mathbb{Q} = \exp \left( M_T - \frac{1}{2} \langle M\rangle_T \right) \cdot \mathbb{P}$$
for the martingale $M = \int_0^\cdot g(s, Y_s, Z_s) \mathrm{d}W_s.$ Supposing that
such a measure $\mathbb{Q}$ is equivalent to the historical measure
$\mathbb{P}$, the classical solution pair $(Y,Z)$ results from projection
and representation respectively, i.e.
\begin{equation}\label{meassol-projectionrepresentation}Y = \mathbb{E}^{\mathbb{Q}}
[\xi|\mathcal{F}_\cdot] = Y_0 + \int_0^\cdot Z_s \mathrm{d}
\tilde{W}_s\end{equation} where $\tilde{W}$ is a Wiener process under
$\mathbb{Q}$. It is known from \cite{meassol-meas} that basically all classical
solutions can be interpreted as measure solutions.\par\medskip

In this paper we view measure solutions still more generally as probability
measures $\mathbb{Q}$ related to terminal variables $\xi$ and
\emph{generating functions} $g$ such that the operation of projection and
representation providing the pair of processes $(Y,Z)$ according to
(\ref{meassol-projectionrepresentation}) leads to an interpretation of the
exponential martingale density $\zeta$ in
$$\frac{\mathrm{d} \mathbb{Q}}{\mathrm{d} \mathbb{P}} = \exp\left(\int_0^T \zeta_s \mathrm{d}W_s -
\frac{1}{2} \int_0^T \zeta_s^2 \mathrm{d}s \right)$$ by
$$\zeta = g(\cdot,Y,Z).$$
Obviously, in case $\mathbb{Q}\sim\mathbb{P}$ this notion allows to
identify $(Y,Z)$ as the classical solution of (\ref{meassol-eq:introduction})
related to the \emph{generator} $f(\cdot,z) = z\cdot g(\cdot,z)$, since
$$Y = Y_0 + \int_0^\cdot Z_s [\mathrm{d}W_s - \zeta_s \mathrm{d}s] = Y_0 + \int_0^\cdot Z_s
\mathrm{d}W_s - \int_0^\cdot f(s,Y_s,Z_s) \mathrm{d}s.$$ Note that a generator $f$ which
is quadratic in $z$ corresponds to a generating function $g$ which is of at
most linear growth in $z$. Our main aim is to provide a result on the
existence of a measure solution (and thus a classical one) in a scenario in
which the terminal variable is bounded, the generating function $g$ is just
continuous off the hyperplane $z=0$ and fulfills the rather general
boundedness hypothesis $|g(s,\cdot,z)| \le C (|z| + \phi_s),$ where the
stochastic integral of $\phi$ with respect to $W$ is a BMO martingale. Note
that in this scenario the bound on $g$ may be random, a detail which turned
out to be of considerable practical relevance for example in \cite{meassol-diff}
and \cite{meassol-cross}. The main tool we develop in order to reach this goal
consists in a stability property for measure solutions. Given a sequence of
measure solutions $(\mathbb{Q}_n)_{n\in\mathbb{N}}$ associated with
terminal conditions and generating functions given by $(\xi_n,
g_n)_{n\in\mathbb{N}}$, we formulate sufficient conditions under which a
limiting measure solution can be found, as a weak limit of the sequence
$(\mathbb{Q}_n)_{n\in\mathbb{N}}$. Given a generating function in the
situation of our main existence theorem, the sequence of approximating
measure solutions is constructed along smoothed approximations of the
generating function, obtained by a new technique based on comparison
properties of classical minimal solutions. Owing to the boundedness
conditions valid for the generating function $g$, BMO martingale techniques
play an important role in our reasoning.\par\medskip

The paper is organized as follows. In section \ref{meassol-prelim}, we explain the
notion of measure solution and a slight modification of it, the
\emph{almost-measure solution}. In section \ref{meassol-stabres}, our main result
on the stability of measure solutions is established in a technical proof
based on tools related to the weak convergence of probability measures and
the martingale representation of their Radon-Nikodym densities (Theorems
\ref{meassol-result} and \ref{meassol-result'}). This result is combined with comparison
related tools in section \ref{meassol-existres} to prove our main statement on the
existence of measure and classical solutions of BSDE (Theorem
\ref{meassol-resunb}). In an appendix we collect some (extensions of) well known
results about BMO martingales, martingale representation and duality in
normed spaces.

\section{Preliminaries} \label{meassol-prelim}

Let $(\Omega,\mathcal{F}_T,\mathbb{P},(\mathcal{F}_t)_{t\in[0,T]})$ be a filtered probability space, such that the filtration satisfies the usual hypotheses. Assume furthermore there exists a $d$-dimensional Brownian motion $W$ on $[0,T]$, which is progressive with respect to $(\mathcal{F}_t)_{t\in[0,T]}$ and such that $\mathcal{F}_t=\mathcal{F}^W_t$, the natural filtration generated by $W$ (and augmented by the null sets).

Define for $q \geq 1$ and any probability measure $\mathbb{Q}$ the set $\mathcal{H}^q(\mathbb{R}^m,\mathbb{Q})$ as the space of all progressive processes $(Z_t)_{t\in[0,T]}$ with values in $\mathbb{R}^m$ normed by $\|Z\|_{\mathcal{H}^q}:=\mathbb{E}_\mathbb{Q}\left[\left(\int_0^T|Z_s|^2\,\mathrm{d} s\right)^\frac{q}{2}\right]^\frac{1}{q}<\infty$.

Let $\mathbb{Q}\sim\mathbb{P}$ be a probability measure. Define
$R_T:=\frac{\mathrm{d}\mathbb{Q}}{\mathrm{d}\mathbb{P}}$ as the Radon-Nikodym
derivative. Then the martingale $R:=\mathbb{E}[R_T|\mathcal{F}_\cdot]$ can be
written as
$R=\exp\left(\int_0^\cdot\zeta_s\,\mathrm{d} W_s-\frac{1}{2}\int_0^\cdot|\zeta_s|^2\,\mathrm{d} s\right)$\footnote{$\zeta_s \,\mathrm{d} W_s=\zeta_s\cdot\,\mathrm{d} W_s$ stands for the scalar
product between the vectors $\zeta_s$ and $\,\mathrm{d} W_s$. It is used as an
abbreviation for the notation $\zeta_s^\top \,\mathrm{d} W_s$.}
with some progressively measurable process $\zeta$ such that $\int_0^T|\zeta_s|^2\,\mathrm{d}s<\infty$ a.s. (Lemma \ref{meassol-density}).\\
Define
\begin{equation}\label{meassol-wq} W^{\mathbb{Q}}:=W-\int_0^\cdot\zeta_s\,\mathrm{d}s. \end{equation}
Then $W^{\mathbb{Q}}$ is a Brownian motion with respect to $\mathbb{Q}$ according to Girsanov's Theorem. \\
It is well known (e.g. Lemma 1.6.7 in \cite{meassol-karat}) that $W^{\mathbb{Q}}$ has the representation property in $(\mathcal{F}_t)_{t\in[0,T]}$, i.e. for any real-valued $\mathcal{F}_T$-measurable $\xi$, which is integrable with respect to $\mathbb{Q}$, we have
\begin{equation}\label{meassol-decomp} \xi=\mathbb{E}_\mathbb{Q}[\xi]+\int_0^T Z_s\,\mathrm{d}W^\mathbb{Q}_s \end{equation}
with some progressively measurable process $Z$ such that $\int_0^T |Z_s|^2\,\mathrm{d}s<\infty$ a.s. \\
Define $Y_t:=\mathbb{E}_{\mathbb{Q}}[\xi|\mathcal{F}_t]$ for all $t\in[0,T]$.

\begin{definition} We call a function
$$ f: \Omega\times[0,T]\times\mathbb{R}^{n}\rightarrow\mathbb{R}^{m} $$
\textup{proper} if $f$ restricted to $\Omega\times[0,t]\times\mathbb{R}^{n}$ is
$\mathcal{F}_t\otimes\mathcal{B}([0,t])\otimes\mathcal{B}(\mathbb{R}^{n})$-measurable for all $t\in[0,T]$.
\end{definition}

If $f$ is proper and $X$ is a progressive $\mathbb{R}^{n}$-valued process,
then the process $(\omega,s)\longmapsto f(\omega,s,X_s(\omega))$ is
progressive as well. This is because the mapping
$$\Omega\times [0,t]\rightarrow \Omega\times[0,t]\times\mathbb{R}^{n}$$
$$(\omega,s)\longmapsto (\omega,s,X_s(\omega))$$
is
$\mathcal{F}_t\otimes\mathcal{B}([0,t])-\mathcal{F}_t\otimes\mathcal{B}([0,t])
\otimes\mathcal{B}(\mathbb{R}^{n})$ - measurable for all $t\in[0,T]$. The
following definition presents the principal concept of this paper.

\begin{definition}
For a given probability measure $\mathbb{Q}\sim\mathbb{P}$ and $\xi\in L^1(\mathcal{F}_T,\mathbb{Q})$, let $\zeta$, $Z$ and $Y$ be as above.
Now let $$ g: \Omega\times[0,T]\times\mathbb{R}\times\mathbb{R}^d\rightarrow\mathbb{R}^d$$ be proper.
We say that $\mathbb{Q}$ is a \emph{measure solution} of the BSDE given by $g$ and $\xi$ if
$$\zeta=g(\cdot,Y,Z) \qquad \mathrm{d} \mathbb{P}\otimes\mathrm{d}t \textrm{ - a.e.} $$
\end{definition}
\begin{remark}
Inserting this definition into (\ref{meassol-wq}) and the result into
(\ref{meassol-decomp}) we have
$$ \xi=\mathbb{E}_\mathbb{Q}[\xi]+\int_0^T Z_s\,\mathrm{d} W_s-\int_0^T Z_s\cdot g(s,Y_s,Z_s) \,\mathrm{d} s. $$
And similarly using $Y=\mathbb{E}_\mathbb{Q}[\xi|\mathcal{F}_\cdot]=Y_0+\int_0^\cdot Z_s\,\mathrm{d}W^\mathbb{Q}_s$ we have
$$ Y=Y_0+\int_0^\cdot Z_s\,\mathrm{d} W_s-\int_0^\cdot Z_s\cdot g(s,Y_s,Z_s) \,\mathrm{d} s. $$
Thus $(Y,Z)$ is also a classical solution of the BDSE given by the generator $f$ satisfying
$$f(s,y,z):=z\cdot g(s,y,z),\quad s\in[0,T], z\in\mathbb{R}^d, y\in\mathbb{R},$$ and the terminal condition $Y_T=\xi$. \footnote{$z\cdot g(s,y,z)$ stands for
the scalar product of $z$ and $g(s,y,z)$, which sometimes will be
abbreviated simply by $z g(s,y,z)$.}
\end{remark}

In the following lemma we characterize generators $f$ which can be written
in the form $f(t,y,z)=z\cdot g(t,y,z), t\in[0,T], y\in\mathbb{R},
z\in\mathbb{R}^d,$ with some continuous $g$.

\begin{lemma} Let $f: \mathbb{R}^d\rightarrow\mathbb{R}$ be continuous. There exists some
continuous $g: \mathbb{R}^d\rightarrow\mathbb{R}^d$ such that for $z\in\mathbb{R}^d$ we
have $f(z)=z\cdot g(z)$ if and only if the following two conditions are satisfied:
\begin{itemize}
\item $f(0)=0$,
\item $f$ is differentiable at $0$.
\end{itemize}
\end{lemma}
\begin{proof}
Let $g$ be continuous such that for $z\in\mathbb{R}^d$ we have $f(z)=z\cdot g(z)$. Then $f(0)=0\cdot g(0)=0$. Furthermore
$$ \frac{1}{|z|}|f(z)-f(0)-z\cdot g(0)|=\frac{1}{|z|}|z\cdot (g(z)-g(0))|\leq |g(z)-g(0)|\rightarrow 0\textrm{ as } z\rightarrow 0 $$
which means that $f$ is differentiable at $0$ with $\nabla_z f(0)=g(0)$.

Conversely let $f$ be differentiable at $0$ and $f(0)=0$. Define $$g(z):=\frac{z}{|z|^2}(f(z)-z\cdot\nabla_z f(0))+\nabla_z f(0)$$
for $z\neq 0$ and $g(0):=\nabla_z f(0)$.
Since
$$ \left|\frac{z}{|z|^2}(f(z)-z\cdot\nabla_z f(0))\right|=\frac{1}{|z|}|f(z)-f(0)-z\cdot\nabla_z f(0)|\rightarrow 0\textrm{ as } z\rightarrow 0 $$
due to differentiability of $f$ at $0$, $g$ is continuous at $0$ and hence
everywhere. Finally we have for $z\in\mathbb{R}^d$
$$ z\cdot g(z)=\frac{|z|^2}{|z|^2}(f(z)-z\cdot\nabla_z f(0))+z\cdot\nabla_z f(0)=f(z). $$
\end{proof}

\begin{definition}
For a given probability measure $\mathbb{Q}\sim\mathbb{P}$ and $\xi\in L^1(\mathcal{F}_T,\mathbb{Q})$ let $\zeta$, $Z$ and $Y$ be defined as in section \ref{meassol-prelim}. Let
$$ g: \Omega\times[0,T]\times\mathbb{R}\times\mathbb{R}^d\rightarrow\mathbb{R}^d $$
be proper. We say that $\mathbb{Q}$ is an \emph{almost measure solution} (a.-measure solution) of the BSDE given by $g$ and $\xi$ if
$$\zeta\mathbf{1}_{\{Z\neq 0\}}=g(\cdot,Y,Z)\mathbf{1}_{\{Z\neq 0\}} \qquad \mathrm{d}\mathbb{P}\otimes\mathrm{d}t \textrm{ - a.e.} $$
\end{definition}

\begin{remark}\label{meassol-alm} Obviously measure solutions are always a.-measure solutions.

Substituting the definition of an a.-measure solution into (\ref{meassol-wq}) and
the result into (\ref{meassol-decomp}) we have
$$ \xi=\mathbb{E}_\mathbb{Q}[\xi]+\int_0^T Z_s\,\mathrm{d} W_s-\int_0^T Z_s\cdot g(s,Y_s,Z_s) \,\mathrm{d} s. $$
And similarly using $Y_t=\mathbb{E}_\mathbb{Q}[\xi|\mathcal{F}_t]=Y_0+\int_0^t Z_s\,\mathrm{d}W^\mathbb{Q}_s$ we have
$$ Y_t=Y_0+\int_0^t Z_s\,\mathrm{d} W_s-\int_0^t Z_s\cdot g(s,Y_s,Z_s) \,\mathrm{d} s, \quad t\in[0,T]. $$
Thus $(Y,Z)$ is also a classical solution of the BDSE given by the
generator $f$ which satisfies
$$f(s,y,z)=z\cdot g(s,y,z),\quad s\in[0,T], y\in\mathbb{R}, z\in\mathbb{R}^d,$$ and the terminal condition $Y_T=\xi$.
\end{remark}

\begin{remark}\label{meassol-alm2}
If there is a proper $\hat{g}$, which might differ from $g$, but satisfies
$z\cdot\hat{g}(\cdot, z)=z\cdot g(\cdot, z)$ and if for
$\hat{\zeta}:=\hat{g}(\cdot,Y,Z)$ the measure
$\hat{\mathbb{Q}}:=\mathcal{E}(\hat{\zeta}\bullet W)_T\cdot\mathbb{P}$ is a
probability measure satisfying
$\mathbb{E}_{\hat{\mathbb{Q}}}[|\xi|]<\infty$, then it must be already a
\emph{measure solution} of the BSDE given by $\hat{g}$ and $\xi$. In fact,
for $t\in[0,T]$ we have
$$ Y_t=\mathbb{E}_\mathbb{Q}[\xi]+\int_0^t Z_s\,\mathrm{d} W_s-\int_0^t Z_s\cdot g(s,Y_s,Z_s) \,\mathrm{d} s= $$
$$ =\mathbb{E}_\mathbb{Q}[\xi]+\int_0^t Z_s\,\mathrm{d} W_s-\int_0^t Z_s\cdot \hat{\zeta}_s \,\mathrm{d} s=
\mathbb{E}_\mathbb{Q}[\xi]+\int_0^t Z_s\,\mathrm{d} W^{\hat{\zeta}}_s $$ where
$W^{\hat{\zeta}}_t:=W_t-\int_0^t\hat{\zeta}_s\,\mathrm{d} s$ is a Brownian motion
with respect to $\hat{\mathbb{Q}}$. Hence
$\hat{Y}_t:=\mathbb{E}_{\hat{\mathbb{Q}}}[\xi|\mathcal{F}_t]=\mathbb{E}_{\hat{\mathbb{Q}}}[Y_T|
\mathcal{F}_t]=\mathbb{E}_\mathbb{Q}[\xi]+\int_0^t Z_s\,\mathrm{d}
W^{\hat{\zeta}}_s=Y_t$ and $\hat{Z}=Z$. Therefore
$\hat{\zeta}=\hat{g}(\cdot,\hat{Y},\hat{Z})$ and so $\hat{\mathbb{Q}}$ is
indeed
a measure solution. \\
In most cases we will set $\hat{g}:=g$.
\end{remark}

\section{Stability Results}\label{meassol-stabres}

In this section we shall consider stability properties of the measure
solution concept. More formally, we shall look at sequences of measure
solutions related to sequences of terminal conditions and generators, and
provide answers to the question: under which additional conditions
concerning these model parameters will an eventually existing weak limit
measure describe a measure solution. We shall find sufficient conditions
for this to hold, including uniform $L^p$-boundedness of the Radon-Nikodym
derivatives of the sequence of measures, and suitable convergence
properties for the sequences of terminal conditions and generators. Our
main results are given by the following two technical theorems.

\begin{theorem}\label{meassol-result}
Let $(\mathbb{Q}_n)_{n\in\mathbb{N}}$ be a sequence of measures, such that
for any $n\in\mathbb{N}$, $\mathbb{Q}_n$ describes a measure solution of the
BSDE given by some proper $g_n$ and $\xi_n\in L^0(\mathcal{F}_T)$ and let
$Y^n$, $Z^n$ correspond to $\mathbb{Q}_n$ and $\xi_n$ in the sense of
section \ref{meassol-prelim}. Let $p,q>1$ such that $1/p+1/q=1$ and assume the following:
\begin{description}
\item{i)}
$$ \sup_n\mathbb{E}\left[\left(\frac{\mathrm{d}\mathbb{Q}_n}{\mathrm{d}\mathbb{P}}\right)^p\right]<\infty \,,
\qquad \sup_n\mathbb{E}_{\mathbb{Q}_n}\left[\left(
\frac{\mathrm{d}\mathbb{P}}{\mathrm{d}\mathbb{Q}_n}\right)^p\right]<\infty;
$$
\item{ii)} $\xi_n\rightarrow\xi$ as $n\rightarrow\infty$ in $L^{2q}(\mathbb{P})$
and $\xi\in L^{q^2}(\mathbb{P})$;
\item{iii)} $Y^n$ converges in $\mathcal{H}^q(\mathbb{R},\mathbb{P})$
to some $\widetilde{Y}\in\mathcal{H}^q(\mathbb{R},\mathbb{P})$ for
$n\rightarrow\infty$;
\item{iv)} $\lim_{n\rightarrow\infty}g_n(\omega,s,\cdot,\cdot)=g(\omega,s,\cdot,\cdot)$
uniformly on compact sets $K\ni(y,z)$ for a.a. $(\omega,s)$, with $g$ proper and continuous
in $(y,z)$ for a.a. $(\omega,s)$.
\end{description}
Then $(\mathbb{Q}_n)_{n\in\mathbb{N}}$ converges "weakly" to a probability
measure $\mathbb{Q}$ in the sense
\begin{equation}\label{meassol-weak} \lim_{n\rightarrow\infty}\mathbb{E}_{\mathbb{Q}_n}[X]=
\mathbb{E}_{\mathbb{Q}}[X] \qquad\forall X\in L^q(\mathbb{P}),
\end{equation} and $\mathbb{Q}$ is a measure solution of the BSDE given by
$g$ and $\xi$.
\end{theorem}

Before we prove this result, let us discuss conditions \emph{i)}-\emph{iv)}: \\
Conditions \emph{ii)} and \emph{iv)} essentially mean that the
approximating BSDE should converge to the limiting BSDE in some sense. \\
Condition \emph{i)} can be seen as a compactness criterion. It stipulates
that the $p$-norms of the Radon-Nikodym-derivatives of the measure change
between $\mathbb{Q}_n$ and $\mathbb{P}$ should be uniformly bounded in $n$
in both "directions". This control is essential to conclude the existence
of a cluster point $\mathbb{Q}$ which
will later be shown to be a limit.\\
Condition \emph{iii)} is necessary to conclude convergence of the sequence
of measure solutions. It can be shown later that $\widetilde{Y}$ must be
the $Y$-process of the measures solution of the limiting BSDE. When
applying the Theorem, condition \emph{iii)} will usually be verified by
choosing for a given pair $(g,\xi)$ the approximating $(g_n,\xi_n)$ in such
a way that the associated generators $f_n(\cdot, z):=z\cdot g_n(\cdot, z),
z\in\mathbb{R}^d,$ as well as the terminal conditions $\xi_n$ are chosen
from a monotonically increasing or decreasing sequence. This allows to
apply the comparison theorem to conclude monotonicity of the $Y^n$, which
in turn will imply their convergence in
$\mathcal{H}^q(\mathbb{R},\mathbb{P})$ by means of dominated convergence.

\begin{proof}[Proof of Theorem \ref{meassol-result}]
For $n\in\mathbb{N}$ let
$R^n_T:=\frac{\mathrm{d}\mathbb{Q}_n}{\mathrm{d}\mathbb{P}}$. Then
$(R^n_T)_{n\in\mathbb{N}}$ is a bounded
sequence in $L^p(\mathbb{P})=L^p(\mathcal{F}_T,\mathbb{P})$ according to \emph{i)}. \\
Firstly we claim, that in order to prove the theorem it is actually
sufficient to show that all subsequences of
$(\mathbb{Q}_n)_{n\in\mathbb{N}}$ possess a particular probability measure
$\mathbb{Q}$ as a cluster point  (with respect to convergence figuring in
(\ref{meassol-weak})) and that this measure
$\mathbb{Q}$ is a measure solution of the BSDE given by $g$ and $\xi$. \\
In fact, assume that this is the case. Let $X\in L^q(\mathbb{P})$. Then
$\mathbb{E}_{\mathbb{Q}_n}[X]=\mathbb{E}[R^n_T X], n\in\mathbb{N},$ is a
bounded sequence by H\"older's inequality. Hence there exists a subsequence
$(\mathbb{Q}_{n_k})_{k\in\mathbb{N}}$ such that
$\mathbb{E}_{\mathbb{Q}_{n_k}}[X]$ converges to the superior limit of the
sequence. But then for a subsequence of
$(\mathbb{Q}_{n_k})_{k\in\mathbb{N}}$ the corresponding subsequence of
$(\mathbb{E}_{\mathbb{Q}_{n_k}}[X])_{k\in\mathbb{N}}$ would converge to
$\mathbb{E}_{\mathbb{Q}}[X]$. Hence $\mathbb{E}_{\mathbb{Q}}[X]$ is equal
to the superior limit of $(\mathbb{E}_{\mathbb{Q}_n}[X])_{n\in\mathbb{N}}$.
Similarly, the inferior limit would also be equal to
$\mathbb{E}_{\mathbb{Q}}[X]$. This concludes the proof.

Let us now show in several steps that all subsequences of
$(\mathbb{Q}_n)_{n\in\mathbb{N}}$ possess a probability measure
$\mathbb{Q}$  as a cluster point, and that this measure is a measure
solution of the BSDE given by $g$ and $\xi$. Since $(R^n_T)_{n\in\mathbb{N}}$
(or any subsequence of $(R^n_T)_{n\in\mathbb{N}}$) is bounded in
$L^p(\mathbb{P})$, there exists an $R_T\in L^p(\mathbb{P})$ such that
$\mathbb{E}[R_T^p]\leq\sup_{n}\mathbb{E}[(R_T^n)^p]$ together with a
subsequence of $(R^n_T)_{n\in\mathbb{N}}$ (or of any subsequence of
$(R^n_T)_{n\in\mathbb{N}}$), which we will again denote by
$(R^n_T)_{n\in\mathbb{N}}$, converging weakly to $R_T$, i.e.
$\lim_{n\rightarrow\infty}\mathbb{E}[R^n_T X]=\mathbb{E}[R_T X]$ for all $X\in
L^q(\mathbb{P}),$
where $\frac{1}{q}+\frac{1}{p}=1$. This holds true since $L^p(\mathbb{P})$ is a reflexive Banach space. \\
>From this we conclude that $\mathbb{E}[R_T X]\geq 0$ for all non-negative and
bounded $X\in L^\infty(\mathcal{F}_T)$, which means that $R_T$ is a.s.
non-negative (setting $X=\mathbf{1}_{\{R_T<0\}}$ we have
$R_T\mathbf{1}_{\{R_T<0\}}=0$ a.s.). Furthermore setting $X=1$ we get
$\mathbb{E}[R_T]=1$. This means that $\mathbb{Q}:=R_T\cdot\mathbb{P}$
is a probability measure. \\
It remains to show that $\mathbb{Q}$ is indeed a measure solution of the
BSDE given by $g$ and $\xi$ and that $\mathbb{Q}$ is  uniquely determined
and does not depend on the subsequence chosen.

\textsc{Claim1:} $\mathbb{Q}$ is equivalent to $\mathbb{P}$ and $\mathbb{E}_{\mathbb{Q}}\left[\left(\frac{\mathrm{d}\mathbb{P}}{\mathrm{d}\mathbb{Q}}\right)^p\right]<\infty$. \\
\emph{Proof:} Using Lemma \ref{meassol-lpnorm} we obtain
$$ \left(\mathbb{E}_{\mathbb{Q}}\left[\left(\frac{1}{R_T}\right)^p\right]\right)^\frac{1}{p}=
\sup_{L^\infty(\mathcal{F}_T)\ni X> 0}
\frac{\mathbb{E}_{\mathbb{Q}}\left[\frac{1}{R_T} X\right]}{\left(\mathbb{E}_{\mathbb{Q}}
\left[|X|^q\right]\right)^{1/q}}=
\sup_{L^\infty(\mathcal{F}_T)\ni X> 0}
\frac{\mathbb{E}\left[X\right]}{\left(\mathbb{E}\left[R_T
|X|^q\right]\right)^{1/q}}= $$
$$ =\sup_{L^\infty(\mathcal{F}_T)\ni X> 0}
\lim_{n\rightarrow\infty}\frac{\mathbb{E}\left[X\right]}{\left(\mathbb{E}\left[R^n_T
|X|^q\right]\right)^{1/q}} \leq \sup_{L^\infty(\mathcal{F}_T)\ni X>
0}\sup_{n}\frac{\mathbb{E}\left[X\right]}{\left(\mathbb{E}\left[R^n_T
|X|^q\right]\right)^{1/q}}=$$
$$=\sup_{n}\sup_{L^\infty(\mathcal{F}_T)\ni X> 0} \frac{\mathbb{E}_{\mathbb{Q}_n}\left[\frac{1}{R^n_T} X\right]}{\left(\mathbb{E}_{\mathbb{Q}_n}\left[|X|^q\right]\right)^{1/q}}= \sup_{n}
\left(\mathbb{E}_{\mathbb{Q}_n}\left[\left(\frac{1}{R^n_T}\right)^p\right]\right)^\frac{1}{p}<\infty.
$$
This means $\mathbb{E}\left[\frac{1}{R^{p-1}_T}\right]=\mathbb{E}
\left[R_T\left(\frac{1}{R_T}\right)^p\right]<\infty,$ and therefore $R_T>0$
$\mathbb{P}$-a.s.,  and thus we have $\mathbb{Q}\sim\mathbb{P}$.
$\hfill\square$ \\

Now for $n\in\mathbb{N}$ set $R^n:=\mathbb{E}[R^n_T|\mathcal{F}_\cdot]$,
write
$$ R^n=\exp\left(\int_0^\cdot\zeta^n_s\mathrm{d}W_s-\frac{1}{2}\int_0^\cdot|\zeta^n_s|^2\mathrm{d}s\right) $$
for some progressively measurable $\zeta^n$, and similarly
$$ R=\exp\left(\int_0^\cdot\zeta_s\mathrm{d}W_s-\frac{1}{2}\int_0^\cdot|\zeta_s|^2\mathrm{d}s\right), $$
with $R:=\mathbb{E}[R_T|\mathcal{F}_\cdot]$. We now claim that $\zeta^n$
converges to $\zeta$ in a weak
sense. \\

\textsc{Claim2:} $$ \lim_{n\rightarrow\infty}\mathbb{E}_{\mathbb{Q}_n}\left[\int_0^T \zeta^n_s\cdot \lambda_s\,\mathrm{d}s\right]=\mathbb{E}_{\mathbb{Q}}\left[\int_0^T \zeta_s\cdot\lambda_s\,\mathrm{d}s\right] $$
for all progressively measurable $\mathbb{R}^d$-valued $\lambda$ such that $\mathbb{E}\left[\left(\int_0^T |\lambda_s|^2 \,\mathrm{d}s\right)^\frac{q}{2}\right]<\infty$.\\
\emph{Proof:} Using It\^{o}'s formula we obtain for $n\in\mathbb{N}$
\begin{equation}\label{meassol-itoo1} R^n_T=1+\int_0^T R^n_s\zeta^n_s\,\mathrm{d}W_s \end{equation}
and
$$ R_T=1+\int_0^T R_s\zeta_s\,\mathrm{d}W_s. $$
Now define $X_T:=\int_0^T \lambda_s\,\mathrm{d}W_s$ and for $t\in[0,T]$ let
$X_t:=\int_0^t \lambda_s\,\mathrm{d}W_s$. Then $\mathbb{E}[R^n_T X_T]$
converges to $\mathbb{E}[R_T X_T]$, since $X_T\in L^q(\mathbb{P})$, which
follows from the BDG inequalities. On the other hand we can calculate the
co-variation of the martingales $R^n$ and $X$ by
$$ \langle R^n,X\rangle_\cdot = \int_0^\cdot R^n_s\zeta^n_s \lambda_s \,\mathrm{d}s. $$
Therefore $$ \mathbb{E}[R^n_{\tau_k}X_{\tau_k}]=\mathbb{E}\left[\int_0^{\tau_k} R^n_s\zeta^n_s \lambda_s \,\mathrm{d}s\right] $$
for some localizing sequence of stopping times $(\tau_k)$. \\
Using (\ref{meassol-itoo1}), $\sup_n\mathbb{E}[(R^n_T)^p]<\infty$ and the BDG
inequalities we have
\begin{equation}\label{meassol-cont1} \sup_n\mathbb{E}\left[\left(\int_0^T |R^n_s\zeta^n_s|^2 \,
\mathrm{d}s\right)^\frac{p}{2}\right]<\infty, \end{equation} which by
Cauchy-Schwarz' and H\"older's inequalities implies $\sup_n
\mathbb{E}\left[\int_0^{T} |R^n_s\zeta^n_s \lambda_s|
\,\mathrm{d}s\right]<\infty$. By dominated convergence for
$k\rightarrow\infty$ this implies for $n\in\mathbb{N}$
$$ \mathbb{E}[R^n_T X_T]=\mathbb{E}\left[\int_0^T R^n_s\zeta^n_s \lambda_s \,\mathrm{d}s\right]. $$
Here we use in particular that as a corollary of Doob's inequality $\sup_{t\in[0,T]}|R^n_t|
\in L^p(\mathbb{P})$ and $\sup_{t\in[0,T]}|X_t| \in L^q(\mathbb{P})$, such that $R^n_T X_T$ is integrable. \\
Similarly
$$ \mathbb{E}[R_T X_T]=\mathbb{E}\left[\int_0^T R_s\zeta_s \lambda_s \,\mathrm{d}s\right]. $$
Finally
$$ \mathbb{E}\left[\int_0^T R^n_s\zeta^n_s \lambda_s \,\mathrm{d}s\right]=
\int_0^T \mathbb{E}[R^n_s\zeta^n_s \lambda_s] \,\mathrm{d}s=\int_0^T \mathbb{E}\Big[ \mathbb{E}\big[R^n_T|\mathcal{F}_s\big]\zeta^n_s \lambda_s \Big] \,\mathrm{d}s= $$
$$ =\int_0^T \mathbb{E}[\mathbb{E}[R^n_T \zeta^n_s \lambda_s |\mathcal{F}_s]] \,\mathrm{d}s =
\mathbb{E}\left[\int_0^T R^n_T \zeta^n_s \lambda_s \,\mathrm{d}s\right]=\mathbb{E}_{\mathbb{Q}_n}\left[\int_0^T \zeta^n_s \lambda_s\,\mathrm{d}s\right] $$
And similarly $\mathbb{E}\left[\int_0^T R_s\zeta_s \lambda_s \,\mathrm{d}s\right]=
\mathbb{E}_{\mathbb{Q}}\left[\int_0^T \zeta_s \lambda_s\,\mathrm{d}s\right]$. \\
Now the assertion follows from the convergence of $\mathbb{E}[R^n_T X_T]$ to
$\mathbb{E}[R_T X_T]$ as $n\to\infty.$
$\hspace*{\fill}\square$ \\

It is not difficult to show that for $n\in\mathbb{N}$, $\xi_n\in L^2(\mathbb{Q}_n)$ using H\"older's inequality and \emph{ii)} as well as the $p$-integrability of $R^n_T$. The same holds for $\xi$ and $\mathbb{Q}$. We define $W^{\mathbb{Q}_n}:=W-\int_0^\cdot\zeta^n_s\mathrm{d}s$ and $W^{\mathbb{Q}}:=W-\int_0^\cdot\zeta_s\mathrm{d}s$ which are Brownian motions with respect to $\mathbb{Q}_n$ and $\mathbb{Q}$ respectively.
We can furthermore write:
$$ \xi_n=\mathbb{E}_{\mathbb{Q}_n}[\xi_n]+\int_0^T Z^n_s\,\mathrm{d}W^{\mathbb{Q}_n}_s, $$
$$ \xi=\mathbb{E}_{\mathbb{Q}}[\xi]+\int_0^T Z_s\,\mathrm{d}W^{\mathbb{Q}}_s, $$
with some progressively measurable $Z^n$ and $Z$.

Now define
$$ Y:=\mathbb{E}_\mathbb{Q}[\xi|\mathcal{F}_\cdot]. $$
In order to complete the proof we have to show $\zeta=g(\cdot,Y,Z)$. This will follow from $\zeta^n=g_n(\cdot,Y^n,Z^n), n\in\mathbb{N},$ using an appropriate kind of convergence of $\zeta^n$ to $\zeta$ and $Z^n$ to $Z$. We will use the type of convergence shown in \textsc{Claim2}. It has the property of uniqueness of limits, as shown in the following.

If $\lim_{n\rightarrow\infty}\mathbb{E}_{\mathbb{Q}_n}\left[\int_0^T a^n_s \lambda_s\,\mathrm{d}s\right]= \mathbb{E}_{\mathbb{Q}}\left[\int_0^T a_s \lambda_s\,\mathrm{d}s\right]$ for all bounded progressively measurable $\lambda$ and at the same time $\lim_{n\rightarrow\infty}\mathbb{E}_{\mathbb{Q}_n}\left[\int_0^T a^n_s \lambda_s\,\mathrm{d}s\right]=\mathbb{E}_{\mathbb{Q}}\left[\int_0^T b_s \lambda_s\,\mathrm{d}s\right]$ for all such $\lambda$, then $a$ and $b$ must be equal $\mathrm{d} \mathbb{P}\otimes\mathrm{d}t$-a.e.\\
This follows from $\mathbb{E}_{\mathbb{Q}}\left[\int_0^T (b_s-a_s) \lambda_s\,\mathrm{d}s\right]=0$ for all bounded $\lambda$ and the equivalence of $\mathbb{Q}$ and $\mathbb{P}$. \\
We will therefore apply the following strategy. \\
In order to complete the proof, recalling \textsc{Claim2}, it is sufficient to show
\begin{equation}\label{meassol-final}  \lim_{n\rightarrow\infty}\mathbb{E}_{\mathbb{Q}_n}\left[\int_0^T g_n(s,Y^n_s,Z^n_s) \lambda_s\,\mathrm{d}s\right]= \mathbb{E}_{\mathbb{Q}}\left[\int_0^T g(s,Y_s,Z_s) \lambda_s\,\mathrm{d}s\right] \end{equation}
for a class of progressive $\lambda$ large enough to be specified later. In order to prove this we will first show
$$ \lim_{n\rightarrow\infty}\mathbb{E}_{\mathbb{Q}_n}\left[\int_0^T Z^n_s \lambda_s\,\mathrm{d}s\right]= \mathbb{E}_{\mathbb{Q}}\left[\int_0^T Z_s \lambda_s\,\mathrm{d}s\right], $$
and then use
$$ \lim_{n\rightarrow\infty}\mathbb{E}_{\mathbb{Q}_n}\left[\int_0^T |Z^n_s|^2\,\mathrm{d}s\right]= \mathbb{E}_{\mathbb{Q}}\left[\int_0^T |Z_s|^2 \,\mathrm{d}s\right], $$
to conclude
$$ \lim_{n\rightarrow\infty}\mathbb{E}_{\mathbb{Q}_n}\left[\int_0^T |Z^n_s-Z_s|^2 \,\mathrm{d}s\right]=0, $$
from which
$$ \lim_{n\rightarrow\infty} g_n(s,Y^n_s,Z^n_s)=g(s,Y_s,Z_s)\quad\textrm{a.e.} $$
will follow after passing to a subsequence. This conclusion, obtained in the following six step argument, will finally imply (\ref{meassol-final}). \\

\textsc{Step1:} For all bounded progressively measurable processes $\lambda$ we have:
$$ \lim_{n\rightarrow\infty}\mathbb{E}_{\mathbb{Q}_n}\left[\int_0^T Z^n_s \lambda_s\,\mathrm{d}s\right]= \mathbb{E}_{\mathbb{Q}}\left[\int_0^T Z_s \lambda_s\,\mathrm{d}s\right]. $$
\emph{Proof:} For $n\in\mathbb{N}$ first write $\xi_n R^n_T$ as
$$ \xi_n R^n_T =\mathbb{E}[\xi_n R^n_T]+\int_0^T \eta^n_s\,\mathrm{d}W_s $$
and define the martingale $V^n:=\mathbb{E}[\xi_n R^n_T]+\int_0^\cdot \eta^n_s\,\mathrm{d}W_s=\mathbb{E}[\xi_n R^n_T|\mathcal{F}_\cdot]$. Using $\xi_n=\frac{V^n_T}{R^n_T}$ and remembering $R^n_T=1+\int_0^T R^n_s\zeta^n_s\,\mathrm{d}W_s$ we can use It\^{o}'s formula to express $Z^n$ in terms of $R^n$, $V^n$, $\zeta^n$ and $\eta^n$ explicitly (e.g. Lemma 1.6.7 in \cite{meassol-karat}). We end up with
$$ Z^n=\frac{\eta^n-V^n\zeta^n}{R^n},\quad n\in\mathbb{N}. $$
A similar formula holds for $Z$, $\zeta$, $R$, $V:=\mathbb{E}[R_T \xi|\mathcal{F}_\cdot]$ and $\eta$ such that $ \xi R_T=\mathbb{E}[\xi R_T]+\int_0^T \eta_s\,\mathrm{d}W_s$. \\
Now let $\lambda$ be as in the claim. Then for $n\in\mathbb{N}$
$$ \mathbb{E}_{\mathbb{Q}_n}\left[\int_0^T Z^n_s \lambda_s\,\mathrm{d}s\right]
=\mathbb{E}\left[\int_0^T (\eta^n_s-V^n_s\zeta^n_s) \lambda_s\,\mathrm{d}s\right]. $$
We first want to show $\lim_{n\rightarrow\infty}\mathbb{E}\left[\int_0^T\eta^n_s \lambda_s\,\mathrm{d}s\right]=
\mathbb{E}\left[\int_0^T\eta_s \lambda_s\,\mathrm{d}s\right]$.
The argument relies on the fact that $\lim_{n\rightarrow\infty}\mathbb{E}[\xi_n R^n_T X]=\mathbb{E}[\xi R_T X]$ for $X:=\int_0^T \lambda_s\,\mathrm{d}W_s$, on the equation $$\mathbb{E}\left[\int_0^T\eta^n_s \lambda_s\,\mathrm{d}s\right]=
\mathbb{E}\left[\left(\int_0^T \eta^n_s\,\mathrm{d}W_s\right)\left(\int_0^T \lambda_s\,\mathrm{d}W_s\right)\right]=\mathbb{E}[\xi_n R^n_T X] $$
valid for $n\in\mathbb{N}$ and a similar one for the limiting processes. This is guaranteed by $\xi_n R^n_T$ being in some $L^{p'}(\mathbb{P})$ with $1<p'$ small enough following from H\"older's inequality and \emph{i)}, \emph{ii)}, which also implies
$ \mathbb{E}\left[\left(\int_0^T |\eta^n_s|^2 \,\mathrm{d}s\right)^{p'/2}\right]<\infty $. \\
The next argument gives
$\lim_{n\rightarrow\infty}\mathbb{E}[\xi_n R^n_T X]=\mathbb{E}[\xi R_T X]$. In fact,
$$ |\mathbb{E}[\xi_n R^n_T X]-\mathbb{E}[\xi R_T X]|\leq |\mathbb{E}[\xi_n R^n_T X]-\mathbb{E}[\xi R^n_T X]|+|\mathbb{E}[\xi R^n_T X]- \mathbb{E}[\xi R_T X]|\leq $$
$$ \leq \mathbb{E}[|\xi_n X -\xi X|^q]^\frac{1}{q} \sup_n\mathbb{E}\left[(R^n_T)^p\right]^\frac{1}{p}+|\mathbb{E}[\xi R^n_T X]- \mathbb{E}[\xi R_T X]|\leq $$
\begin{equation}\label{meassol-argu} \leq \mathbb{E}[|\xi_n-\xi|^{2q}]^\frac{1}{2q}\mathbb{E}[|X|^{2q}]^\frac{1}{2q}
\sup_n\mathbb{E}\left[(R^n_T)^p\right]^\frac{1}{p}+|\mathbb{E}[R^n_T \xi X]- \mathbb{E}[R_T \xi X]|\longrightarrow 0 \end{equation}
for $n\to\infty.$ Here we employed the convergence of $\xi_n$ to $\xi$ in $L^{2q}(\mathbb{P})$ as well as $X\in L^{r}(\mathbb{P})$ for all $r\geq 1$. Furthermore we used the weak convergence of $R^n_T$ to $R_T$ and $\xi X \in L^q(\mathbb{P})$.

Secondly we have to show $\lim_{n\rightarrow\infty}\mathbb{E}\left[\int_0^T V^n_s\zeta^n_s \lambda_s\,\mathrm{d}s\right]=
\mathbb{E}\left[\int_0^T V_s\zeta_s \lambda_s\,\mathrm{d}s\right]$. \\
Applying Bayes' formula for $n\in\mathbb{N}$ we have $Y^n=\frac{\mathbb{E}[R^n_T\xi_n|\mathcal{F}_\cdot]}{R^n}$ and thus
$$ \mathbb{E}\left[\int_0^T V^n_s\zeta^n_s \lambda_s\,\mathrm{d}s\right] =
\mathbb{E}\left[\int_0^T \mathbb{E}[R^n_T\xi_n|\mathcal{F}_s]\zeta^n_s \lambda_s\,\mathrm{d}s\right]
=\mathbb{E}\left[\int_0^T Y^n_s R^n_s\zeta^n_s \lambda_s\,\mathrm{d}s\right]= $$
$$ =\mathbb{E}_{\mathbb{Q}_n}\left[\int_0^T Y^n_s\zeta^n_s \lambda_s\,\mathrm{d}s\right]. $$
An analogous equation holds for the limiting processes.
We now proceed in two steps: First we show that the difference between $\mathbb{E}_{\mathbb{Q}_n}\left[\int_0^T Y^n_s\zeta^n_s \lambda_s\,\mathrm{d}s\right]$ and $\mathbb{E}_{\mathbb{Q}_n}\left[\int_0^T Y_s\zeta^n_s \lambda_s\,\mathrm{d}s\right]= \mathbb{E}\left[\int_0^T Y_sR^n_s\zeta^n_s \lambda_s\,\mathrm{d}s\right]$ converges to zero and then we show that
$\mathbb{E}_{\mathbb{Q}_n}\left[\int_0^T Y_s\zeta^n_s \lambda_s\,\mathrm{d}s\right]$ converges to $\mathbb{E}_{\mathbb{Q}}\left[\int_0^T Y_s\zeta_s \lambda_s\,\mathrm{d}s\right]$. \\
Using Cauchy-Schwarz' and H\"older's inequalities, as well as $$\sup_{n}\mathbb{E}\left[\left(\int_0^T |R^n_s\zeta^n_s \lambda_s|^2\,\mathrm{d}s\right)^\frac{p}{2}\right]<\infty$$ (see (\ref{meassol-cont1})), the first step would follow from
$$ \lim_{n\rightarrow\infty}\mathbb{E}\left[\left(\int_0^T |Y^n_s-Y_s|^2 \,\mathrm{d}s\right)^\frac{q}{2}\right]=0. $$
This holds essentially because of condition \emph{iii)}, according to which $Y^n$ has a limit in the above sense. It remains to verify that $Y=\widetilde{Y}$. For this purpose it will be sufficient to show that $Y^n$ converges both to $Y$ and $\widetilde{Y}$ in a weak sense. Setting $X=\int_0^T \mu_s\,\mathrm{d}s$ for some bounded progressively measurable $\mu$ we have
$$ \mathbb{E}_{\mathbb{Q}_n}\left[\int_0^T Y^n_s \mu_s\,\mathrm{d}s\right]=
\mathbb{E}\left[\int_0^T \mathbb{E}[R^n_T\xi_n|\mathcal{F}_s] \mu_s\,\mathrm{d}s\right] =
\mathbb{E}\left[R^n_T \xi_n \int_0^T \mu_s\,\mathrm{d}s\right]=$$
$$=\mathbb{E}[R^n_T \xi_n X] \longrightarrow \mathbb{E}\left[R_T \xi X\right]=\ldots=
\mathbb{E}_{\mathbb{Q}}\left[\int_0^T Y_s \mu_s\,\mathrm{d}s\right] \textrm{ as }n\rightarrow\infty. $$
For the limit we used the reasoning of (\ref{meassol-argu}). \\
On the other hand $Y^n$ must converge to $\widetilde{Y}$ in the same sense. Indeed we have for $n\in\mathbb{N}$
$$ \left|\mathbb{E}_{\mathbb{Q}_n}\left[\int_0^T Y^n_s \mu_s\,\mathrm{d}s\right]-\mathbb{E}_{\mathbb{Q}}\left[\int_0^T \widetilde{Y}_s \mu_s\,\mathrm{d}s\right]\right|=$$
$$=\left|\mathbb{E}\left[R^n_T\int_0^T Y^n_s \mu_s\,\mathrm{d}s\right]-\mathbb{E}\left[R_T\int_0^T \widetilde{Y}_s \mu_s\,\mathrm{d}s\right]\right|\leq $$
$$ \leq \left|\mathbb{E}\left[R^n_T\int_0^T (Y^n_s-\widetilde{Y}_s) \mu_s\,\mathrm{d}s\right]\right|+\left|\mathbb{E}\left[R_T\int_0^T \widetilde{Y}_s \mu_s\,\mathrm{d}s\right]-\mathbb{E}\left[R^n_T\int_0^T \widetilde{Y}_s \mu_s\,\mathrm{d}s\right]\right|\leq $$
$$ \leq \sup_n\mathbb{E}\left[(R^n_T)^p\right]^\frac{1}{p}\left(\mathbb{E}\left[\left(\int_0^T (Y^n_s-\widetilde{Y}_s)^2 \,\mathrm{d}s\right)^\frac{q}{2}\left(\int_0^T |\mu_s|^2 \,\mathrm{d}s\right)^\frac{q}{2}\right]\right)^\frac{1}{q}+ $$
$$ +\left|\mathbb{E}\left[R_T\int_0^T \widetilde{Y}_s \mu_s\,\mathrm{d}s\right]-\mathbb{E}\left[R^n_T\int_0^T \widetilde{Y}_s \mu_s\,\mathrm{d}s\right]\right| $$
This first summand converges to zero by \emph{i)}, \emph{iii)} and the boundedness of $\mu$. The second one converges to zero as well due to the weak convergence of $R^n_T$ to $R_T$, for which we need $\mathbb{E}\left[\left(\int_0^T |\widetilde{Y}_s|^2\,\mathrm{d}s\right)^\frac{q}{2}\right]<\infty$. \\
Therefore $Y$ and $\widetilde{Y}$ must coincide owing to uniqueness of weak limits. Thus we have shown that the difference between $\mathbb{E}_{\mathbb{Q}_n}\left[\int_0^T Y^n_s\zeta^n_s \lambda_s\,\mathrm{d}s\right]$ and $\mathbb{E}_{\mathbb{Q}_n}\left[\int_0^T Y_s\zeta^n_s \lambda_s\,\mathrm{d}s\right]$ converges to zero. \\
It remains to note that the difference between $\mathbb{E}_{\mathbb{Q}_n}\left[\int_0^T Y_s\zeta^n_s \lambda_s\,\mathrm{d}s\right]$ and $\mathbb{E}_{\mathbb{Q}}\left[\int_0^T Y_s\zeta_s \lambda_s\,\mathrm{d}s\right]$ tends to zero as well. This is a consequence of \textsc{Claim2}. \\
Here we used $\mathbb{E}\left[\left(\int_0^T |Y_s|^2\,\mathrm{d}s\right)^\frac{q}{2}\right]=\mathbb{E}\left[\left(\int_0^T |\widetilde{Y}_s|^2\,\mathrm{d}s\right)^\frac{q}{2}\right]<\infty$.
$\hfill\square$ \\

\textsc{Step2:} We have
$$ \lim_{n\rightarrow\infty}\mathbb{E}_{\mathbb{Q}_n}\left[\int_0^T |Z^n_s|^2\,\mathrm{d}s\right]= \mathbb{E}_{\mathbb{Q}}\left[\int_0^T |Z_s|^2 \,\mathrm{d}s\right]. $$
\emph{Proof:} This follows from It\^{o}'s isometry based on $\lim_{n\rightarrow\infty}\mathbb{E}_{\mathbb{Q}_n}[\xi_n]=\mathbb{E}_{\mathbb{Q}}[\xi]$ and $\lim_{n\rightarrow\infty}\mathbb{E}_{\mathbb{Q}_n}[\xi_n^2]=\mathbb{E}_{\mathbb{Q}}[\xi^2]$. Both statements follow from the strong convergence of $\xi_n$ to $\xi$ stated in \emph{ii)} and the weak convergence of $R^n_T$ to $R_T$. We only give details for the last one. In fact
$$ |\mathbb{E}_{\mathbb{Q}_n}[\xi_n^2]-\mathbb{E}_{\mathbb{Q}}[\xi^2]|=|\mathbb{E}[R^n_T\xi_n^2]-\mathbb{E}[R^n_T \xi^2]|+
|\mathbb{E}[R^n_T\xi^2]-\mathbb{E}[R_T\xi^2]|\leq $$
$$ \leq \mathbb{E}[|\xi_n^2-\xi^2|^q]^\frac{1}{q}\sup_n\mathbb{E}\left[(R^n_T)^p\right]^\frac{1}{p}+|\mathbb{E}[R^n_T\xi^2]-\mathbb{E}[R_T\xi^2]|\leq $$
$$ \leq \mathbb{E}[|\xi_n+\xi|^{2q}]^\frac{1}{2q} \mathbb{E}[|\xi_n-\xi|^{2q}]^\frac{1}{2q}\sup_n\mathbb{E}\left[(R^n_T)^p\right]^\frac{1}{p}+|\mathbb{E}[R^n_T\xi^2]-\mathbb{E}[R_T\xi^2]|\longrightarrow 0 $$
as $n\to\infty.$ Here we used \emph{ii)}, $\sup_n\mathbb{E}[|\xi_n|^{2q}]<\infty$ and $\xi^2\in L^q(\mathbb{P})$. \\

\textsc{Step3:} We claim that, eventually passing to a subsequence, we can assume w.l.o.g.
$$ \lim_{n\rightarrow\infty}\mathbb{E}_{\mathbb{Q}_n}\left[\int_0^T |Z^n_s-Z_s|^2 \,\mathrm{d}s\right]=0. $$

\emph{Proof:} First of all, as a consequence of the BDG inequalities, \emph{ii)} and the definition of $Z$ we get $\mathbb{E}_{\mathbb{Q}}\left[\left(\int_0^T |Z_s|^2 \,\mathrm{d}s\right)^{q^2}\right]<\infty$. Moreover, $\mathbb{E}_{\mathbb{Q}}[(1/R_T)^p]<\infty$ combined with H\"older's inequality implies
\begin{equation}\label{meassol-cont2} \mathbb{E}\left[\left(\int_0^T |Z_s|^2 \,\mathrm{d}s\right)^q\right]=\mathbb{E}_\mathbb{Q}\left[\frac{1}{R_T}\left(\int_0^T |Z_s|^2 \,\mathrm{d}s\right)^q\right]<\infty \end{equation}
which in turn implies
$$ \sup_n\mathbb{E}_{\mathbb{Q}_n}\left[\int_0^T |Z_s|^2 \,\mathrm{d}s\right]<\infty, $$
again by H\"older's inequality. \\
Using $|Z^n_s-Z_s|^2=|Z^n_s|^2-2Z^n_sZ_s+|Z_s|^2$ and \textsc{Step2}, we see that it is sufficient to show
$$ \lim_{n\rightarrow\infty}\mathbb{E}_{\mathbb{Q}_n}\left[\int_0^T 2Z^n_sZ_s\,\mathrm{d}s\right]= \mathbb{E}_{\mathbb{Q}}\left[\int_0^T 2Z_sZ_s\,\mathrm{d}s\right] $$
which obviously must follow from an extension of \textsc{Step1}. The difficulty we have to overcome here is that $Z$ is not necessarily bounded. However by replacing $Z$ with some $\tilde{Z}^k:=((-k)\vee Z)\wedge k$  we have for any $k\in\mathbb{N}$
$$ \lim_{n\rightarrow\infty}\mathbb{E}_{\mathbb{Q}_n}\left[\int_0^T Z^n_s\tilde{Z}^k_s\,\mathrm{d}s\right]= \mathbb{E}_{\mathbb{Q}}\left[\int_0^T Z_s\tilde{Z}^k_s\,\mathrm{d}s\right]. $$
Since $\lim_{k\rightarrow\infty}\mathbb{E}_{\mathbb{Q}}\left[\int_0^T Z_s\tilde{Z}^k_s\,\mathrm{d}s\right]=\mathbb{E}_{\mathbb{Q}}\left[\int_0^T |Z_s|^2\,\mathrm{d}s\right]$ using dominated convergence it remains to show
$$ \lim_{n\rightarrow\infty}\mathbb{E}_{\mathbb{Q}_n}\left[\int_0^T Z^n_sZ_s\,\mathrm{d}s\right]=\lim_{k\rightarrow\infty}\lim_{n\rightarrow\infty}\mathbb{E}_{\mathbb{Q}_n}\left[\int_0^T Z^n_s\tilde{Z}^k_s\,\mathrm{d}s\right]. $$
Since by Cauchy-Schwarz' inequality $\mathbb{E}_{\mathbb{Q}_n}\left[\int_0^T Z^n_sZ_s\,\mathrm{d}s\right], n\in\mathbb{N},$ is a bounded sequence, we can assume w.l.o.g. that it is convergent, after eventually passing to a subsequence. Hence
$$ \lim_{k\rightarrow\infty}\lim_{n\rightarrow\infty}\mathbb{E}_{\mathbb{Q}_n}\left[\int_0^T Z^n_s\tilde{Z}^k_s\,\mathrm{d}s\right]= \qquad\qquad\qquad\qquad\qquad\qquad\qquad\qquad $$ $$\qquad\qquad\qquad\qquad=\lim_{k\rightarrow\infty}\lim_{n\rightarrow\infty}\left(\mathbb{E}_{\mathbb{Q}_n}\left[\int_0^T Z^n_s(\tilde{Z}^k_s-Z_s)\,\mathrm{d}s\right]+\mathbb{E}_{\mathbb{Q}_n}\left[\int_0^T Z^n_s Z_s\,\mathrm{d}s\right]\right). $$
Therefore it remains to show
$$ \lim_{k\rightarrow\infty}\lim_{n\rightarrow\infty}\mathbb{E}_{\mathbb{Q}_n}\left[\int_0^T Z^n_s(\tilde{Z}^k_s-Z_s)\,\mathrm{d}s\right]=0. $$
We have
$$ \sup_{n}\left|\mathbb{E}_{\mathbb{Q}_n}\left[\int_0^T Z^n_s(\tilde{Z}^k_s-Z_s)\,\mathrm{d}s\right]\right|^2\leq \hspace*{5cm} $$
$$\hspace*{2cm} \leq \sup_{n}\mathbb{E}_{\mathbb{Q}_n}\left[\int_0^T |Z^n_s|^2 \,\mathrm{d}s\right]\sup_{n}\mathbb{E}_{\mathbb{Q}_n}\left[\int_0^T |\tilde{Z}^k_s-Z_s|^2 \,\mathrm{d}s\right]\leq $$
$$ \leq \sup_{n}\mathbb{E}_{\mathbb{Q}_n}\left[\int_0^T |Z^n_s|^2 \,\mathrm{d}s\right] \sup_{n}\mathbb{E}[(R^n_T)^p]^\frac{1}{p}\left(\mathbb{E}\left[\left(\int_0^T |\tilde{Z}^k_s-Z_s|^2 \,\mathrm{d}s\right)^q\right]\right)^\frac{1}{q}. $$
The result follows from $\tilde{Z}^k\rightarrow Z$, dominated convergence and (\ref{meassol-cont2}).
$\hfill\square$ \\

\textsc{Step4:} Passing to a subsequence once more, we deduce
$$ \lim_{n\rightarrow\infty}g_n(\cdot,Y^n,Z^n)=g(\cdot,Y,Z) \qquad \mathrm{d} \mathbb{P}\otimes\mathrm{d}s\textrm{ a.e.}$$
\emph{Proof:}
We have
$$ \mathbb{E}\left[\int_0^T |Z^n_s-Z_s|^2\wedge\frac{1}{T} \,\mathrm{d}s\right]=
\mathbb{E}_{\mathbb{Q}_n}\left[\frac{1}{R^n_T}\int_0^T |Z^n_s-Z_s|^2\wedge\frac{1}{T} \,\mathrm{d}s\right]\leq $$
$$ \leq\left(\sup_n\mathbb{E}_{\mathbb{Q}_n}\left[\left(\frac{1}{R^n_T}\right)^p\right]\right)^\frac{1}{p}\left(\mathbb{E}_{\mathbb{Q}_n}\left[\left(\int_0^T |Z^n_s-Z_s|^2\wedge\frac{1}{T} \,\mathrm{d}s\right)^q\right]\right)^\frac{1}{q}\leq $$
$$ \leq\left(\sup_n\mathbb{E}_{\mathbb{Q}_n}\left[\left(\frac{1}{R^n_T}\right)^p\right]\right)^\frac{1}{p}\left(\mathbb{E}_{\mathbb{Q}_n}\left[\int_0^T |Z^n_s-Z_s|^2 \,\mathrm{d}s\right]\right)^\frac{1}{q}\rightarrow 0 \textrm{ as }n\rightarrow\infty. $$
Hence by choosing a subsequence which we again denote by the same symbols we have $|Z^n-Z|^2\wedge\frac{1}{T} \rightarrow 0$ a.e., which means $Z^n\rightarrow Z$ a.e. We can also assume $Y^n\rightarrow Y=\widetilde{Y}$ a.e. This means in particular that $(Y^n_s(\omega),Z^n_s(\omega))$ is a bounded sequence for a.a. $(\omega,s)$. Hence by \emph{iv)}
$$ \lim_{n\rightarrow\infty}|g_n(\omega,s,Y^n_s(\omega),Z^n_s(\omega))-g(\omega,s,Y^n_s(\omega),Z^n_s(\omega))|=0 $$
for a.a. $(\omega,s)$. We also have using the continuity of $g$
$$ \lim_{n\rightarrow\infty}|g(\omega,s,Y^n_s(\omega),Z^n_s(\omega))-g(\omega,s,Y_s(\omega),Z_s(\omega))|=0 $$
for a.a. $(\omega,s)$. This proves the assertion.
$\hfill\square$ \\

\textsc{Step5:} For all bounded progressively measurable $\lambda$ such that $\sup_n |g_n(s,Y^n_s,Z^n_s) \lambda_s|\leq C_\lambda$ with some constant $C_\lambda\in\mathbb{R}$ depending only on $\lambda$ we have:
$$ \lim_{n\rightarrow\infty}\mathbb{E}_{\mathbb{Q}_n}\left[\int_0^T g_n(s,Y^n_s,Z^n_s) \lambda_s\,\mathrm{d}s\right]= \mathbb{E}_{\mathbb{Q}}\left[\int_0^T g(s,Y_s,Z_s) \lambda_s\,\mathrm{d}s\right]. $$
\emph{Proof:} Since $\mathbb{E}_{\mathbb{Q}_n}\left[\int_0^T g(s,Y_s,Z_s) \lambda_s\,\mathrm{d}s\right]$ converges to
$\mathbb{E}_{\mathbb{Q}}\left[\int_0^T g(s,Y_s,Z_s) \lambda_s\,\mathrm{d}s\right]$ according to the weak convergence of $R^n_T$ to $R_T$ and the boundedness of $|g(s,Y_s,Z_s) \lambda_s|$, it is enough to show
$$ \mathbb{E}_{\mathbb{Q}_n}\left[\int_0^T g_n(s,Y^n_s,Z^n_s) \lambda_s\,\mathrm{d}s\right]-\mathbb{E}_{\mathbb{Q}_n}\left[\int_0^T g(s,Y_s,Z_s) \lambda_s\,\mathrm{d}s\right]\rightarrow 0. $$
But this follows from \textsc{Step4} by noting
$$ \mathbb{E}_{\mathbb{Q}_n}\left[\int_0^T (g_n(s,Y^n_s,Z^n_s)-g(s,Y_s,Z_s)) \lambda_s\,\mathrm{d}s\right]\leq $$
$$ \leq \sup_n\mathbb{E}[(R^n_T)^p]^\frac{1}{p} \left(\mathbb{E}\left[\left(\int_0^T (g_n(s,Y^n_s,Z^n_s)-g(s,Y_s,Z_s)) \lambda_s\,\mathrm{d}s\right)^q\right]\right)^\frac{1}{q} $$
which tends to zero for $n\rightarrow\infty$ by dominated convergence using the uniform boundedness of $g_n(\cdot,Y^n,Z^n) \lambda$ and hence $g(\cdot,Y,Z) \lambda$.
$\hfill\square$ \\

\textsc{Step6:} Using \textsc{Claim2}, \textsc{Step5} and $\zeta^n=g_n(\cdot,Y^n,Z^n), n\in\mathbb{N},$ we have
$$ \mathbb{E}_{\mathbb{Q}}\left[\int_0^T (\zeta_s-g(s,Y_s,Z_s))\lambda_s\,\mathrm{d}s\right]=0 $$
for all bounded progressive $\lambda$ such that $\sup_n |g_n(\cdot,Y^n,Z^n) \lambda|\leq C_\lambda$. Now define for all $C>0$
$$ \lambda^C:=(\zeta-g(\cdot,Y,Z))\mathbf{1}_{\{|\zeta-g(\cdot,Y,Z)|\leq C\}}\mathbf{1}_{\{\sup_{n}|g_n(\cdot,Y^n,Z^n)|\leq C\}} $$
which is progressive, bounded and satisfies $|g_n(\cdot,Y^n,Z^n)\lambda^C|\leq C^2$. This implies
$$ \mathbb{E}_{\mathbb{Q}}\left[\int_0^T |\zeta_s-g(s,Y_s,Z_s)|^2\mathbf{1}_{\{|\zeta_s-g(s,Y_s,Z_s)|\leq C\}}\mathbf{1}_{\{\sup_{n}|g_n(s,Y^n_s,Z^n_s)|\leq C\}}\,\mathrm{d}s\right]=0, $$
and hence
$$ |\zeta-g(\cdot,Y,Z)|^2\mathbf{1}_{\{|\zeta-g(\cdot,Y,Z)|\leq C\}}\mathbf{1}_{\{\sup_{n}|g_n(\cdot,Y^n,Z^n)|\leq C\}}=0 \qquad \mathrm{d}\mathbb{P}\otimes\mathrm{d}t-\textrm{ a.e.} $$
Clearly $\sup_{n}|g_n(\omega,s,Y^n_s(\omega),Z^n_s(\omega))|$, $g(\omega,s,Y_s(\omega),Z_s(\omega))$ and $\zeta_s(\omega)$ are real numbers for almost all $(\omega,s)$. Hence letting $C$ go to infinity we have $\zeta=g(\cdot,Y,Z)$ $\mathrm{d} \mathbb{P}\otimes\mathrm{d}s\textrm{-a.e.}$, which means that $\mathbb{Q}$ is a measure solution of the BSDE given by $g$ and $\xi$. \\
It remains to mention that although in our construction $\mathbb{Q}$ depends on the choice of a weakly convergent subsequence of $(R^n_T)_{n\in\mathbb{N}}$, we always end up with the same $\mathbb{Q}$, since it is already determined by $\tilde{Y}=Y=\mathbb{E}_\mathbb{Q}[\xi|\mathcal{F}_\cdot]$. Once the $Y$-process of a measure solution $\mathbb{Q}$ is given, the $Z$-process is uniquely determined by the martingale part of $Y$ and this already determines
$\mathbb{Q}=\mathcal{E}(g(\cdot,Y,Z)\bullet W)_T\cdot\mathbb{P}$.
\end{proof}

\begin{remark}\label{meassol-resrem1}
In particular we have shown
$$\mathbb{E}\left[\left(\frac{\mathrm{d}\mathbb{Q}}{\mathrm{d}\mathbb{P}}\right)^p\right]\leq\sup_n\mathbb{E}\left[\left(\frac{\mathrm{d}\mathbb{Q}_n}{\mathrm{d}\mathbb{P}}\right)^p\right] \,,
\qquad \mathbb{E}_{\mathbb{Q}}\left[\left(\frac{\mathrm{d}\mathbb{P}}{\mathrm{d}\mathbb{Q}}\right)^p\right]\leq\sup_n\mathbb{E}_{\mathbb{Q}_n}\left[\left(\frac{\mathrm{d}\mathbb{P}}{\mathrm{d}\mathbb{Q}_n}\right)^p\right]. $$
\end{remark}
\begin{remark}\label{meassol-resrem2}
We have also shown
$\tilde{Y}=\mathbb{E}_{\mathbb{Q}}[\xi|\mathcal{F}_\cdot]$
$\mathrm{d}\mathbb{P}\otimes\mathrm{d}t$-a.e.
\end{remark}
\begin{remark}\label{meassol-resrem3}
In addition we have proven that each subsequence of
$(Z^n)_{n\in\mathbb{N}}$ has a subsequence which converges to $Z$
$\mathbb{P}\otimes\lambda_{[0,T]}$-a.e. In other words
$(Z^n)_{n\in\mathbb{N}}$ converges to $Z$ in measure.
\end{remark}

In the following theorem, the result of Theorem \ref{meassol-result} will be
refined and extended to a situation in which the generating sequence
$(g_n)_{n\in\mathbb{N}}$ converges to $g$ only uniformly on compacts
avoiding the origin in $\mathbb{R}^d$ and in which accordingly only
a.-measure solutions are involved.

\begin{theorem}\label{meassol-result'}
Let $(\mathbb{Q}_n)_{n\in\mathbb{N}}$ be a sequence of measures, each
giving an a.-measure solution of the BSDE given by some $g_n$ and $\xi_n$.
Let $p,q>1$ such that $1/p+1/q=1$ and such that:
\begin{description}
\item{i)}
$$ \sup_n\mathbb{E}\left[\left(\frac{\mathrm{d}\mathbb{Q}_n}{\mathrm{d}\mathbb{P}}\right)^p\right]<\infty \,,
\qquad \sup_n\mathbb{E}_{\mathbb{Q}_n}\left[\left(\frac{\mathrm{d}\mathbb{P}}{\mathrm{d}\mathbb{Q}_n}\right)^p\right]<\infty, $$
\item{ii)} $\xi_n\rightarrow\xi$ as $n\rightarrow\infty$ in $L^{2q}(\mathbb{P})$ and $\xi\in L^{q^2}(\mathbb{P})$,
\item{iii)} $Y^n:=\mathbb{E}_{\mathbb{Q}_n}[\xi_n|\mathcal{F}_\cdot]$ converges for $n\rightarrow\infty$ to some $\widetilde{Y}\in\mathcal{H}^q(\mathbb{R},\mathbb{P})$ a.e. and in $\mathcal{H}^q(\mathbb{R},\mathbb{P})$,
\item{iv)} for a.a. $(\omega,s)$: $\lim_{n\rightarrow\infty}g_n(\omega,s,\cdot,\cdot)=g(\omega,s,\cdot,\cdot)$
(with a proper $g$) uniformly on compacts $K$ such that
$\{(y,z)|\,y\in\mathbb{R},z=0\in\mathbb{R}^d\}\cap K=\emptyset$. Let
furthermore $g(\omega,s,\cdot,\cdot)$ be continuous outside
$\{(y,z)|\,y\in\mathbb{R},z=0\in\mathbb{R}^d\}$ for a.a. $(\omega,s)$.
\end{description}
Then there exists an a.-measure solution of the BSDE given by $g$ and $\xi$.
\end{theorem}
\begin{proof}
The proof is very similar to the proof of Theorem \ref{meassol-result}. We again
start by passing to a subsequence in order to obtain a weak cluster point.
The proof differs beginning with \textsc{Step4}, since
$\lim_{n\rightarrow\infty}g_n(\cdot,Y^n,Z^n)=g(\cdot,Y,Z)$ will not hold in
general. Instead we can show
$$ \lim_{n\rightarrow\infty}g_n(\cdot,Y^n,Z^n)\mathbf{1}_{\{Z^n\neq 0\}}\mathbf{1}_{\{Z\neq 0\}}=g(\cdot,Y,Z)
\mathbf{1}_{\{Z\neq 0\}} \qquad \mathrm{d}\mathbb{P}\otimes\mathrm{d}t-\textrm{ a.e.}$$ For this purpose, using \textsc{Step2} (see \textsc{Step4} in the proof of Theorem \ref{meassol-result} for
details), we can assume w.l.o.g. that $Z^n\rightarrow Z$ a.e. Hence the
sequence $(Z^n_s(\omega))$ will for a.a. $(\omega,s)$ such that $Z_s(\omega)\neq
0$ be contained in some compact set not containing $0$ for $n$ large
enough. Therefore
$$ \lim_{n\rightarrow\infty}|g_n(s,Y^n_s,Z^n_s)(\omega)\mathbf{1}_{\{Z^n_s\neq 0\}}(\omega)\mathbf{1}_{\{Z_s\neq 0\}}(\omega)- \hspace*{5cm}$$
$$\hspace*{5cm} -g(s,Y^n_s,Z^n_s)(\omega)\mathbf{1}_{\{Z^n_s\neq 0\}}(\omega)\mathbf{1}_{\{Z_s\neq 0\}}(\omega)|=0 $$
for a.a. $(\omega,s)$. We also have, by the continuity of $g$ outside
$\{(y,z)|\,y\in\mathbb{R},z=0\}$:
$$ \lim_{n\rightarrow\infty}|g(\omega,s,Y^n_s(\omega),Z^n_s(\omega))\mathbf{1}_{\{Z^n_s\neq 0\}}(\omega)\mathbf{1}_{\{Z_s\neq 0\}}(\omega)- \hspace*{4cm}$$
$$\hspace*{5cm} -g(\omega,s,Y_s(\omega),Z_s(\omega))\mathbf{1}_{\{Z_s\neq 0\}}(\omega)|=0 $$
for a.a. $(\omega,s)$. This proves the assertion.

Using this we can replace \textsc{Step5} in the proof of Theorem
\ref{meassol-result} by the statement
$$ \lim_{n\rightarrow\infty}\mathbb{E}_{\mathbb{Q}_n}\left[\int_0^T g_n(s,Y^n_s,Z^n_s)\mathbf{1}_{\{Z^n_s\neq 0\}}\mathbf{1}_{\{Z_s\neq 0\}} \lambda_s\,\mathrm{d}s\right]= \hspace*{4cm}$$
$$\hspace*{5cm} =\mathbb{E}_{\mathbb{Q}}\left[\int_0^T g(s,Y_s,Z_s)\mathbf{1}_{\{Z_s\neq 0\}} \lambda_s\,\mathrm{d}s\right] $$
for all bounded progressively measurable $\lambda$ such that $\sup_n
|g_n(\cdot,Y^n,Z^n)\mathbf{1}_{\{Z^n\neq 0\}}\mathbf{1}_{\{Z\neq 0\}}
\lambda|\leq C_\lambda$ a.e. with some constant $C_\lambda>0$
depending only on $\lambda$.

On the other hand, using \textsc{Claim2}, we have
$$ \lim_{n\rightarrow\infty}\mathbb{E}_{\mathbb{Q}_n}\left[\int_0^T \zeta^n_s\mathbf{1}_{\{\inf_{m\geq N}|Z^m_s|>0\}} \mathbf{1}_{\{Z_s\neq 0\}}\lambda_s\,\mathrm{d}s\right]= \hspace*{4cm}$$
$$\hspace*{5cm} =\mathbb{E}_{\mathbb{Q}}\left[\int_0^T \zeta_s\mathbf{1}_{\{\inf_{m\geq N}|Z^m_s|>0\}} \mathbf{1}_{\{Z_s\neq 0\}}\lambda_s\,\mathrm{d}s\right] $$
for all progressive bounded $\lambda$ and all $N\in\mathbb{N}$. Considering
$$g_n(\cdot,Y^n,Z^n)\mathbf{1}_{\{Z^n\neq 0\}}\mathbf{1}_{\{\inf_{m\geq N}|Z^m|>0\}}= \hspace*{5cm}$$
$$\hspace*{4cm} =\zeta^n\mathbf{1}_{\{Z^n\neq 0\}}\mathbf{1}_{\{\inf_{m\geq N}|Z^m|>0\}}=
\zeta^n\mathbf{1}_{\{\inf_{m\geq N}|Z^m|>0\}}$$ for all $n\geq N$, this
implies
$$ \mathbb{E}_{\mathbb{Q}}\Bigg[\int_0^T (\zeta_s\mathbf{1}_{\{\inf_{m\geq N}|Z^m_s|>0\}}
\mathbf{1}_{\{Z_s\neq 0\}}-\hspace*{5cm} $$
\begin{equation}\label{meassol-ugly}\hspace*{3cm}-g(s,Y_s,Z_s)\mathbf{1}_{\{\inf_{m\geq
N}|Z^m_s|>0\}}\mathbf{1}_{\{Z_s\neq 0\}}) \lambda_s\,\mathrm{d}s\Bigg]=0 \end{equation}
for all bounded $\lambda$ such that
$\sup_n |g_n(\cdot,Y^n,Z^n)\mathbf{1}_{\{Z^n\neq 0\}}\mathbf{1}_{\{Z\neq
0\}} \lambda|\leq C_\lambda$. Now define $C>0$
$$ \lambda^C:=(\zeta-g(\cdot,Y,Z))\mathbf{1}_{\{|\zeta-g(\cdot,Y,Z)|\leq C\}}\mathbf{1}_{\{\sup_{n}
|g_n(\cdot,Y^n,Z^n)\mathbf{1}_{\{Z^n\neq 0\}}\mathbf{1}_{\{Z\neq 0\}} |\leq
C\}} $$ which is progressive, bounded and satisfies
$|g_n(\cdot,Y^n,Z^n)\mathbf{1}_{\{Z^n\neq 0\}}\mathbf{1}_{\{Z\neq 0\}}
\lambda^C|\leq C^2$. Inserting this into (\ref{meassol-ugly}) yields
$$ |\zeta-g(\cdot,Y,Z)|\mathbf{1}_{\{Z\neq 0\}} \mathbf{1}_{\{\inf_{m\geq N}|Z^m|>0\}}
\mathbf{1}_{\{|\zeta-g(\cdot,Y,Z)|\leq C\}}\cdot\qquad\qquad\qquad\qquad$$
$$\qquad\qquad\qquad\qquad\qquad\qquad\cdot\mathbf{1}_{\{\sup_{n}|g_n(\cdot,Y^n,Z^n)
\mathbf{1}_{\{Z^n\neq 0\}}\mathbf{1}_{\{Z\neq 0\}} |\leq C\}}=0. $$ If we
let $C,N\rightarrow\infty$, we finally obtain
$$\zeta\mathbf{1}_{\{Z\neq 0\}}=g(\cdot,Y,Z)\mathbf{1}_{\{Z\neq 0\}} \qquad\mathrm{d}\mathbb{P}\otimes\mathrm{d}t\textrm{ a.e.} $$
taking into account that $Z_s(\omega)\neq 0$ implies that $Z^n_s(\omega)$
will lie in a compact set not containing $0$ for $n$ large enough for a.a.
$(\omega,s)$ and then $g_n(s,Y^n_s,Z^n_s)(\omega)$ will be bounded and the
sequence $(Z^n_s(\omega))$ will be bounded away from $0$.

We have thus proven that $\mathbb{Q}=\mathcal{E}(\zeta\bullet W)_T\cdot\mathbb{P}$ is an a.-measure solution.
\end{proof}
\begin{remark}\label{meassol-resrem'1}
As in Theorem \ref{meassol-result} it follows that
$$\mathbb{E}\left[\left(\frac{\mathrm{d}\mathbb{Q}}{\mathrm{d}\mathbb{P}}\right)^p\right]\leq\sup_n\mathbb{E}\left[\left(\frac{\mathrm{d}\mathbb{Q}_n}{\mathrm{d}\mathbb{P}}\right)^p\right] \,,
\qquad
\mathbb{E}_{\mathbb{Q}}\left[\left(\frac{\mathrm{d}\mathbb{P}}{\mathrm{d}\mathbb{Q}}\right)^p\right]
\leq\sup_n\mathbb{E}_{\mathbb{Q}_n}\left[\left(\frac{\mathrm{d}\mathbb{P}}{\mathrm{d}\mathbb{Q}_n}
\right)^p\right]. $$
\end{remark}
\begin{remark}\label{meassol-resrem'2}
We have also shown
$\widetilde{Y}=\mathbb{E}_{\mathbb{Q}}[\xi|\mathcal{F}_\cdot]$
$\mathrm{d}\mathbb{P}\otimes\mathrm{d}t$-a.e.
\end{remark}
\begin{remark}\label{meassol-resrem'3}
In addition we have proven that each subsequence of $(Z_n)_{n\in\mathbb{N}}$ has a subsequence which converges to $Z$, $\mathrm{d}\mathbb{P}\otimes\mathrm{d}t$-a.e. (where $Z$ is already uniquely determined by the martingale part of $\widetilde{Y}$). In other words $(Z_n)_{n\in\mathbb{N}}$ converges to $Z$ in measure.
\end{remark}

\section{Existence Results}\label{meassol-existres}

Equipped with the stability properties of the preceding section, we can now
return to the question of existence of measure solutions, and thus to
solutions of the associated BSDE. Using regularization techniques we will
derive existence results in scenarios where the generator function $g$ is
continuous off the origin in $\mathbb{R}^d$ and roughly of at most linear
growth in $z$, i.e. the classical generator $f(\cdot, y,z) = z\cdot
g(\cdot, y,z), y\in\mathbb{R}, z\in\mathbb{R}^d$ is subquadratic in $z$. We
will therefore consider proper functions
$$ g: \Omega\times[0,T]\times\mathbb{R}\times\mathbb{R}^d\rightarrow\mathbb{R}^d $$
and $f(\cdot, y,z) = z\cdot g(\cdot, y,z), y\in\mathbb{R},
z\in\mathbb{R}^d$.

\begin{lemma}\label{meassol-boundbound'}
Let $\xi\in L^\infty(\mathcal{F}_T)$ and let $g$ be proper such that
\begin{description}
\item{i)} $g$ and $f$ are both bounded by some constant $M$,
\item{ii)} $g(\omega,s,\cdot)$ is continuous at all points $(y,z)$ such that $z\neq 0$ for a.a.
$(\omega,s)$,
\item{iii)} $f(\omega,s,\cdot)$ is uniformly continuous for a.a.
$(\omega,s)$.
\end{description}
Then there exists a measure solution of the BSDE given by $g$ and $\xi$.
\end{lemma}
\begin{proof}
\textsc{Step1} We regularize the generator. To do this, for $\varepsilon>0,
y\in\mathbb{R}, z\in\mathbb{R}^d$ let
$$ \tilde{g}_\varepsilon(\cdot,y,z):=\int_{\mathbb{R}\times\mathbb{R}^d}
g(\cdot,y-\hat{y},z-\hat{z})\rho_\varepsilon(\hat{y},\hat{z})\,\mathrm{d}
(\hat{y},\hat{z}) = g\ast \rho_\varepsilon (\cdot,y,z) $$ with
$$ \rho_\varepsilon (\hat{y},\hat{z}):=\frac{1}{(2\pi\varepsilon)^\frac{d+1}{2}}\exp\left(-\frac{1}{2\varepsilon}\left(|\hat{y}|^2+|\hat{z}|^2\right)\right) $$
Now define for $y\in\mathbb{R}$ and $z\in\mathbb{R}^d$
$$\tilde{f}_\varepsilon(\cdot,y,z):=z\tilde{g}_\varepsilon(\cdot,y,z).$$ We
claim that $\lim_{\varepsilon\rightarrow
0}\tilde{f}_\varepsilon(\omega,s,\cdot,\cdot)=f$ uniformly for almost all
fixed $(\omega, s)$. Indeed, since $f(\omega,s,\cdot,\cdot)$ is uniformly
continuous and bounded, $\lim_{\varepsilon\rightarrow 0}(f\ast
\rho_\varepsilon)(\omega,s,\cdot,\cdot)=f(\omega,s,\cdot,\cdot)$ uniformly
and
$$ |f\ast \rho_\varepsilon(\cdot,y,z)-z\tilde{g}_\varepsilon(\cdot,y,z)|\leq
M \int_{\mathbb{R}\times\mathbb{R}^d} |\hat{z}|\rho_\varepsilon(\hat{y},
\hat{z})\,\mathrm{d} (\hat{y},\hat{z}) \rightarrow 0 \textrm{ as }
\varepsilon\rightarrow 0. $$
Hence $\lim_{\varepsilon\rightarrow 0}\tilde{f}_\varepsilon(\omega,s,\cdot,\cdot)=f(\omega,s,\cdot,\cdot)$ uniformly as well. \\

\textsc{Step2} We now construct Lipschitz continuous and uniformly bounded
approximating sequences $(g_n)_{n\in\mathbb{N}}$ resp.
$(f_n)_{n\in\mathbb{N}}$ of $g$ resp. $f$. To this end, for any
$\varepsilon>0, s\in[0,T], y\in\mathbb{R}, z\in\mathbb{R}^d$ let
$$ \delta^{(\varepsilon)}_s(\omega):=\sup_{y,z}|\tilde{f}_{\varepsilon}(\omega,s,y,z)-f(\omega,s,y,z)|. $$
We have $\lim_{\varepsilon\rightarrow 0}\delta^{(\varepsilon)}_s(\omega)=0$
for a.a. $(\omega,s)$ and, by dominated convergence, we also have
$$\lim_{\varepsilon\rightarrow
0}\mathbb{E}\left[\int_0^T\left(\delta^{(\varepsilon)}_s\right)^4\,\mathrm{d}
s\right]=0.$$
Thus, we can choose a sequence
$(\varepsilon_n)_{n\in\mathbb{N}}$ of positive numbers converging to zero
such that 
$$\mathbb{E}\left[\int_0^T\left(\delta^{(\varepsilon_n)}_s\right)^4\,\mathrm{d}
s\right]\leq 2^{-n}$$ and therefore
$$ \sum _{n=0}^\infty p(n)\mathbb{E}\left[\int_0^T\left(\delta^{(\varepsilon_n)}_s\right)^l\,\mathrm{d} s\right]<\infty $$
for all polynomials $p$ and $l\in\{1,2,3,4\}$. In particular $\sum_{n=1}^\infty p(n)\int_0^T \delta^{(\varepsilon_n)}_s\,\mathrm{d} s<\infty$ a.s. for all polynomials $p$. \\
Set for $n\in\mathbb{N}, (s,y,z)\in[0,T]\times\mathbb{R}\times\mathbb{R}^d$
$$g_n(s,y,z):=\tilde{g}_{\varepsilon_n}(s,y+\sum_{k=0}^\infty \int_0^s (k+1)
\delta^{(\varepsilon_{n+k})}_r\,\mathrm{d} r,z),$$ and accordingly
$$f_n(s,y,z):=\tilde{f}_{\varepsilon_n}(s,y+\sum_{k=0}^\infty \int_0^s
(k+1)\delta^{(\varepsilon_{n+k})}_r\,\mathrm{d} r,z).$$ Defining
$\alpha^{(n)}_s:=\sum_{k=0}^\infty \delta^{(\varepsilon_{n+k})}_s$ and
$\beta^{(n)}_s:=\sum_{k=0}^\infty
(k+1)\delta^{(\varepsilon_{n+k})}_s=\sum_{l=n}^\infty \alpha^{(l)}_s$ we
have
$$ f_n(s,y,z)=\tilde{f}_{\varepsilon_n}(s,y+\int_0^s \beta^{(n)}_r\,\mathrm{d} r,z)\leq
f(s,y+\int_0^s \beta^{(n)}_r\,\mathrm{d} r,z)+\delta^{(\varepsilon_{n})}_s\leq $$
$$ \leq\tilde{f}_{\varepsilon_{n+1}}(s,y+\int_0^s \beta^{(n)}_r\,\mathrm{d} r,z)+ \delta^{(\varepsilon_{n})}_s+\delta^{(\varepsilon_{n+1})}_s \leq $$
$$\leq f_{n+1}(s,y+\int_0^s \beta^{(n)}_r\,\mathrm{d} r-\int_0^s \beta^{(n+1)}_r\,\mathrm{d} r,z)+\sum_{k=0}^\infty \delta^{(\varepsilon_{n+k})}_s = $$
$$= f_{n+1}(s,y+\int_0^s \alpha^{(n)}_r\,\mathrm{d} r,z)+\alpha^{(n)}_s. $$
This means
$$ f_n(s,y,z)\leq f_{n+1}(s,y+\int_0^s \alpha^{(n)}_r\,\mathrm{d} r,z)+\alpha^{(n)}_s, $$
a form of monotonicity we will use later. We have also shown
$$ f_n(s,y,z)\leq f(s,y+\int_0^s \beta^{(n)}_r\,\mathrm{d} r,z)+\beta^{(n)}_s, \quad s\in[0,T], y\in\mathbb{R}, z\in\mathbb{R}^d. $$
Note in addition that $\mathbb{E}\left[\left(\int_0^T \beta^{(n)}_s\,\mathrm{d}
s\right)^4\right]<\infty$ for all $n$ since by H\"older's inequality
$$ \left(\int_0^T \beta^{(0)}_s\,\mathrm{d} s\right)^4\leq T^3 \int_0^T \left(\beta^{(0)}_s\right)^4\,\mathrm{d} s, $$
and again by H\"older's inequality
$$ \left(\sum_{k=0}^\infty (k+1)\delta^{(\varepsilon_{k})}_s\right)^4\leq
\left(\sum_{k=0}^\infty \left(\frac{1}{(k+1)^{3/2}}\right)^{4/3}\right)^3
\left(\sum_{k=0}^\infty \left((k+1)^\frac{3}{2}(k+1)\delta^{(\varepsilon_{k})}_s\right)^4\right)$$
and $\mathbb{E}\left[\int_0^T \sum_{n=0}^\infty (n+1)^{10}\left(\delta^{(\varepsilon_{n})}_s\right)^4\,\mathrm{d} s\right]<\infty$. \\
Since $\sum_{n=1}^\infty \int_0^T n\delta^{(\varepsilon_{n})}_s\,\mathrm{d} s<\infty$ a.s.
we have $\lim_{n\rightarrow\infty} \int_0^s \beta^{(n)}_r\,\mathrm{d} r=0$ for a.a. $(\omega, s)$. \\
Furthermore, since $g(\omega,s,\cdot,\cdot)$ is uniformly continuous on compact sets,
which are disjoint from 
$$\left\{\left(y,z\right)\,\bigg|\,y\in\mathbb{R},z=0\in\mathbb{R}^d\right\}$$
and $\tilde{g}_{\varepsilon}(\omega,s,\cdot,\cdot)$ converges for $\varepsilon\rightarrow 0$
uniformly to $g(\omega,s,\cdot,\cdot)$ on such compacts, $g_n(\omega,s,\cdot,\cdot)$ converges
uniformly to $g$ on such compacts as well (for a.a. $(\omega,s)$) as $n\to\infty$. \\
In addition, for $n\in\mathbb{N}$, $f_n$ is Lipschitz continuous in $(y,z)$
according to the definition of $\tilde{g}_\epsilon$ and
$\tilde{f}_\epsilon$. In fact, for $s\in[0,T], y\in\mathbb{R},
z\in\mathbb{R}^d$
$$ \tilde{f}_\varepsilon(s,y,z)= $$
$$ =\int_{\mathbb{R}\times\mathbb{R}^d} (z-\hat{z})g(s,y-\hat{y},z-\hat{z})\rho_\varepsilon(\hat{y},\hat{z})\,\mathrm{d} (\hat{y},\hat{z})+
\int_{\mathbb{R}\times\mathbb{R}^d} \hat{z}g(s,y-\hat{y},z-\hat{z})\rho_\varepsilon(\hat{y},\hat{z})\,\mathrm{d} (\hat{y},\hat{z})= $$
$$=\int_{\mathbb{R}\times\mathbb{R}^d} f(s,\hat{y},\hat{z})\rho_\epsilon(y-\hat{y},z-\hat{z})\,\mathrm{d} (\hat{y},\hat{z})+
\int_{\mathbb{R}\times\mathbb{R}^d} g(s,\hat{y},\hat{z})(z-\hat{z})\rho_\varepsilon(y-\hat{y},z-\hat{z})\,\mathrm{d} (\hat{y},\hat{z})= $$
$$ = f\ast \rho_\varepsilon (s,y,z)+(g\ast (z\rho_\varepsilon)) (s,y,z), $$
which means that $\tilde{f}_\varepsilon$ is differentiable with respect to to $(y,z)$ with uniformly bounded derivatives, since $f,g$ and the derivatives of $\rho_\varepsilon$ and $z\rho_\varepsilon$ are all uniformly bounded. Furthermore $\tilde{f}_\varepsilon$ and therefore $f_n$ are uniformly bounded. \\
Similarly $g_n$ is Lipschitz continuous and uniformly bounded for
$n\in\mathbb{N}$, because the $\tilde{g}_\varepsilon, \varepsilon>0,$
possess this property.

\textsc{Step3} We now apply Theorem \ref{meassol-result'} to show the existence of
some a.-measure solution $\mathbb{Q}$ of the BSDE given by $g$ and $\xi$.
Since for $n\in\mathbb{N}$ the approximations $f_n,g_n$ are Lipschitz
continuous there exists a measure solution $\mathbb{Q}_n$ of the BSDE given
by $g_n$ and $\xi-\int_0^T \beta^{(n)}_s\,\mathrm{d} s=:\xi_n$. This is guaranteed
by a slight extension (including $y$ as a variable in the generator) of the
relationship between classical and measure solution explained in
\cite{meassol-meas} in the case of Lipschitz continuous generator. We check the
validity of the hypotheses of Theorem \ref{meassol-result'}.

Condition \emph{i)}:

For $n\in\mathbb{N}$ define
$$ R_n:=\exp\left(\int_{0}^T g_n(s,Y^n_s,Z^n_s)\,\mathrm{d} W_s-\frac{1}{2}\int_{0}^T |g_n(s,Y^n_s,Z^n_s)|^2
\,\mathrm{d} s\right). $$ We have to show
$$ \sup_n\mathbb{E}[R_n^p]+\sup_n\mathbb{E}[R_n^{1-p}]<\infty.$$
But this holds for all $p>1$ as a consequence of the uniform boundedness of
$g_n$.

Condition \emph{ii)} holds for $q=2$ by dominated convergence.

Condition \emph{iii)}:

It is sufficient to show monotonicity of $(Y^n)_{n\in\mathbb{N}}$ and boundedness of
$\left(\|Y^n\|_{\mathcal{H}^2(\mathbb{R},\mathbb{P})}\right)_{n\in\mathbb{N}}$.
The latter follows from the boundedness of $\left(\|\xi_n\|_{L^2(\mathbb{R},\mathbb{P})}\right)_{n\in\mathbb{N}}$. We will now show montonicity using the comparison theorem: \\
For $n\in\mathbb{N}$ define $X^n:=Y^n+\int_0^\cdot \alpha^{(n)}_r\,\mathrm{d} r$.
Since
$$ Y^n=\xi_n+\int_\cdot^T f_n(r,Y^n_r,Z^n_r)\,\mathrm{d} r-\int_\cdot^T Z^n_r\,\mathrm{d} W_r, $$
we also have
$$ X^n=\xi_n+\int_0^T \alpha^{(n)}_r\,\mathrm{d} r+\int_\cdot^T \left(f_n(r,Y^n_r,Z^n_r)-\alpha^{(n)}_r\right)\,\mathrm{d} r-
\int_\cdot^T Z^n_r\,\mathrm{d} W_r. $$ Note $\xi_n+\int_0^T \alpha^{(n)}_r\,\mathrm{d}
r=\xi_{n+1}$ and for $r\in[0,T]$
$$f_n(r,Y^n_r,Z^n_r)-\alpha^{(n)}_r\leq f_{n+1}(r,Y^n_r+\int_0^r \alpha^{(n)}_v\,\mathrm{d} v,Z^n_r)=f_{n+1}(r,X^n_r,Z^n_r)$$
according to arguments in \textsc{Step2}. If we now consider the two
standard Lipschitz BSDE given by the generators defined by
$\varphi(r,y,z):=f_n(r,Y^n_r,Z^n_r)-\alpha^{(n)}_r$ and
$\psi(r,y,z)=f_{n+1}(r,y,z), r\in[0,T], y\in\mathbb{R}, z\in\mathbb{R}^d,$
and by the identical terminal condition $\xi_{n+1}$, we see that the first
one is solved (in the classical sense) by $(X^n,Z^n)$ and the second one by
$(Y^{n+1},Z^{n+1})$, and at the same time
$\varphi(r,X^n_r,Z^n_r)\leq\psi(r,X^n_r,Z^n_r), r\in[0,T]$. The comparison
theorem implies $X^n\leq Y^{n+1}$ a.e., and therefore $Y^n\leq Y^{n+1}$
a.e.

Condition \emph{iv)}: Its validity has been discussed in \textsc{Step2}.

Thus Theorem \ref{meassol-result'} is applicable, and there exists an a.-measure
solution of the BSDE given by $\xi$ and $g$. But since $g$ is bounded,
there must exist a measure solution $\check{\mathbb{Q}}$ with the same $Y$
and $Z$, according to Remark \ref{meassol-alm2}.
\end{proof}

\begin{corollary}\label{meassol-minisol}
The measure solution $\check{\mathbb{Q}}$ constructed above is minimal in
the following sense: for all classical solutions $(\tilde{Y},\tilde{Z})$ of
the BSDE corresponding to ($\tilde{f},\tilde{\xi}$) where $\tilde{f}$ is
any bounded generator such that $f\leq\tilde{f}$ and $\xi\leq\tilde{\xi}\in
L^2(\mathbb{P})$ we have $Y\leq\tilde{Y}$ a.e. with
$Y:=\mathbb{E}_{\check{\mathbb{Q}}}[\xi|\mathcal{F}_\cdot]=\lim_{n\rightarrow\infty}Y^n$
from the above construction.
\end{corollary}
\begin{proof}
For $n\in\mathbb{N}$ define $W^n:=\tilde{Y}-\int_0^\cdot \beta^{(n)}_r \,\mathrm{d}
r$ and conclude that $(W^n,\tilde{Z})$ solves (in the classical sense) the
BSDE given by the generator
$\varphi(r,y,z):=\tilde{f}(r,\tilde{Y}_r,\tilde{Z}_r)+\beta^{(n)}_r,
r\in[0,T], y\in\mathbb{R}, z\in\mathbb{R}^d,$ and
$\tilde{\xi}-\int_0^T\beta^{(n)}_r\,\mathrm{d} r=:\tilde{\xi}_n$. We know for
$s\in[0,T], y\in\mathbb{R}, z\in\mathbb{R}^d$
$$ f_n(s,y,z)\leq f(s,y+\int_0^s \beta^{(n)}_r\,\mathrm{d} r,z)+\beta^{(n)}_s $$
and therefore for $r\in[0,T]$ $f_n(r,W^n_r,\tilde{Z}_r)\leq
f(r,\tilde{Y}_r,\tilde{Z}_r)+\beta^{(n)}_r\leq\varphi(r,W^n_r,\tilde{Z}_r)$.
Also note $\xi_n\leq\tilde{\xi}_n$. By the comparison theorem this implies
$Y^n\leq W^n$ a.e. This in turn entails the a.e. relation
$$Y=\lim_{n\rightarrow\infty}Y^n\leq\lim_{n\rightarrow\infty}W^n=\tilde{Y}.$$
\end{proof}

As a consequence of this, minimal solutions are unique.

\begin{corollary}\label{meassol-compa'} Minimal measure solutions according to Corollary \ref{meassol-minisol} are unique in the sense
that $(Y,Z)$ are uniquely determined. \\ Furthermore the following
comparison property holds. For bounded $\xi\leq\hat{\xi}$ and for
generators $f(\cdot, y,z)=z\cdot g(\cdot,y,z)$ and
$\hat{f}(\cdot,y,z)=z\cdot \hat{g}(\cdot,y,z), y\in\mathbb{R},
z\in\mathbb{R}^d$, such that the pairs $(g,f)$ and $(\hat{g},\hat{f})$ both
satisfy conditions \emph{i)}-\emph{iii)} of Theorem \ref{meassol-boundbound'} and
such that $f\leq \hat{f}$ pointwise, the corresponding minimal measure solutions
satisfy $Y\leq \hat{Y}$ a.e.
\end{corollary}

In the following lemma, we obtain a stronger result than in Lemma
\ref{meassol-boundbound'}. In fact, we are able to drop the uniform continuity
condition \emph{iii)} of Lemma \ref{meassol-boundbound'}.

\begin{lemma}\label{meassol-boundbound}
Let $\xi\in L^\infty(\mathcal{F}_T)$ be bounded by some constant $K$. Let
$g$ be proper, let $g$ and $f(\cdot,y,z)=z\cdot g(\cdot,y,z),
y\in\mathbb{R}, z\in\mathbb{R}^d,$ be uniformly bounded and
let $g(\omega,s,\cdot,\cdot)$ be continuous at all points $(y,z)$ such that $z\neq 0$ for a.a. $(\omega,s)$. \\
Then there exists a measure solution of the BSDE given by $g$ and $\xi$.
\end{lemma}
\begin{proof}
For $s\in[0,T], y\in\mathbb{R}, z\in\mathbb{R}^d$ define $\tilde{f}(s,y,z):=f(s,(-K)\vee y\wedge K,z)$
and $\tilde{g}(s,y,z):=g(s,(-K)\vee y\wedge K,z)$. $\tilde{g}$ is proper and $g(\omega,s,\cdot,\cdot)$
is continuous at all points $(y,z)$ such that $z\neq 0$ for a.a. $(\omega,s)$. A measure solution of the BSDE
given by $\tilde{g}$ and $\xi$ will be a measure solution of the BSDE given by $g$ and $\xi$ as well,
since $\tilde{Y}$ is bounded by $K$. Hence w.l.o.g we can assume $f=\tilde{f}$ and $g=\tilde{g}$. \\
Approximate $f$, which is obviously continuous, from below by $f_n,
n\in\mathbb{N},$ which is uniformly bounded and uniformly continuous in
$(y,z)$ and such that $f_n(\cdot,y,z)=f(\cdot,y,z)$ for $|y|\leq K, |z|\leq 1,$
and such that $(f_n)_{n\in\mathbb{N}},$ is pointwise increasing and
converges to $f$ uniformly on compact sets $S$ for a.a. $(\omega,s)$. One
possible definition for $(f_n)_{n\in\mathbb{N}}$ is given by
\begin{equation}\label{meassol-infconvolution}
f_n(\cdot,y,z):=\inf_{\substack{\hat{z}\in\mathbb{R}^d\\\hat{y}\in[-K,K]}}
\left\{f(\cdot,\hat{y},\hat{z})+\frac{n}{0\vee(|z|-n)\wedge 1}
\left|{y\choose z}-{\hat{y}\choose\hat{z}}\right|\right\},\end{equation}
for all $z\in\mathbb{R}^d$, $y\in\mathbb{R}$, $n\in\mathbb{N}$,
where we use the convention $\frac{n}{0}\cdot 0:=0$ and $\frac{n}{0}\cdot r:=\infty$ for $r>0$. See  Lemma \ref{meassol-uniformcontinuity} for details.

Now define $g_n(\cdot,z):=\frac{z}{|z|^2}\mathbf{1}_{\{z\neq0\}}f_n(\cdot,z), z\in\mathbb{R}^d,
n\in\mathbb{N}$. We have $g_n(\cdot,0)=0$. \\
Note that $(g_n)_{n\in\mathbb{N}}$ is uniformly bounded. In fact, for $|z|\leq 1,
|y|\le K$ we have: $|g_n(\cdot,y,z)|=
\frac{1}{|z|}|f_n(\cdot,y,z)|=\frac{1}{|z|}|f(\cdot,y,z)|=\frac{1}{|z|}|z\cdot
g(\cdot,y,z)|\leq |g(\cdot,y,z)|$ which is bounded. Furthermore
$g_n(\omega,s,\cdot,\cdot)$ converges to
$\hat{g}(\omega,s,\cdot,\cdot):=\frac{z}{|z|^2}\mathbf{1}_{\{z\neq0\}}f(\omega,s,\cdot,\cdot)$
uniformly on compacts not containing points with vanishing $z$-part for
a.e. $(\omega, s)$, because of the convergence
of $f_n$ to $f$. \\
Considering the BSDEs given by $g_n$ and $\xi$ there exist, according to
Lemma \ref{meassol-boundbound'} corresponding (minimal) measure solutions
$\mathbb{Q}_n$. According to Corollary \ref{meassol-compa'} for the corresponding
$Y^n$ we have the inequality $Y^n\leq Y^{n+1}$ a.e.. Furthermore
$(Y^n)_{n\in\mathbb{N}}$ is uniformly bounded by $K$. Hence condition
\emph{iii)} of Theorem \ref{meassol-result} is satisfied. The remaining three are
easy to check. Theorem \ref{meassol-result} is therefore applicable, and there
exists an a.-measure solution of the BSDE given by $\xi$ and $\hat{g}$.
However, since $\hat{g}$ is bounded and for $z\in\mathbb{R}^d$ we have
$z\hat{g}(\cdot,z)=zg(\cdot,z)$, there must exist a measure solution of the
BSDE given by $\xi$ and $g$ as well, according to Remark \ref{meassol-alm2}.
\end{proof}

\begin{corollary}\label{meassol-compa} The solution $(\mathbb{Q},Y)$ constructed in Lemma \ref{meassol-boundbound}
is minimal in the following sense. For bounded $\check{\xi}\geq \xi$ and
$\check{f}(\cdot,z)=z\cdot \check{g}(\cdot,z), z\in\mathbb{R}^d$, such that
the pair $(\check{g},\check{f})$ satisfies the conditions of Lemma
\ref{meassol-boundbound} and such that $f\leq \check{f}$ pointwise, we have for a
measure solution $\check{\mathbb{Q}}$: $Y\leq \check{Y}$ a.e., where
$\check{Y}$ corresponds to $\check{\mathbb{Q}}$ and $\check{\xi}$ in the
sense of section \ref{meassol-prelim}.
\end{corollary}
\begin{proof} This is easily shown by means of the approximating generators $f_n$
from the proof of Lemma \ref{meassol-boundbound} satisfying $f_n\leq \check{f},
n\in\mathbb{N}$. This implies $Y^n\leq\check{Y}$ by Corollary \ref{meassol-compa'}.
Therefore
\makebox{$Y=\lim_{n\rightarrow\infty}Y^n\leq\check{Y}$}.\end{proof}

In the following main existence theorem of this paper the statement of the
preceding lemma is extended to generating functions $g$ that are not
uniformly bounded. We are able to treat cases in which the upper bound on
$|g(s,\cdot,z)|$ is at most proportional to $|z|$ and provided by a
progressive process $\phi$ such that its stochastic integral with respect
to $W$ generates a BMO martingale. More prescisely, we have

\begin{theorem}\label{meassol-resunb}
Let $\xi\in L^\infty(\mathcal{F}_T)$. Let $g$ be proper, assume that
$g(\omega,s,\cdot,\cdot)$ is continuous at all points $(y,z)$ such that $z\neq
0$ for a.a. $(\omega,s)$, and $|g(s,\cdot,z)|\leq C(|z|+\phi_s),
s\in[0,T], z\in\mathbb{R}^d,$ with some progressive $\phi\in BMO(\mathbb{P})$ (see Appendix \ref{meassol-appe}) and a constant $C>0$. \\
Then there exists a measure solution of the BSDE given by $g$ and $\xi$.
\end{theorem}
\begin{proof}
Note that $f(\omega,s,y,z):=z\cdot g(\omega,s,y,z)$ is continuous in
$(y,z)\in\mathbb{R}\times\mathbb{R}^d$ for a.a.
$(\omega,s)\in\Omega\times[0,T]$, where the continuity at points $(y,0)$
comes from $|f(s,y,z)|\leq C|z|(|z|+\phi_s)$.

For $z\in\mathbb{R}^d$ define
$\tilde{g}(\cdot,z):=\frac{z}{|z|^2}\mathbf{1}_{\{z\neq 0\}}f(\cdot,z)$
(where $\frac{z}{|z|^2}\mathbf{1}_{\{z\neq 0\}}$ is defined to be $0$ at
$z=0$). Note that $\tilde{g}$ also satifies $|\tilde{g}(s,\cdot,z)|\leq
C(|z|+\phi_s), s\in[0,T], z\in\mathbb{R}^d$.

For $n,m\in\mathbb{N}, s\in[0,T], y\in\mathbb{R}, z\in\mathbb{R}^d$ define
$$ f_{nm}(s,y,z):=(-n)\vee\left(|z|\cdot\,(-n)\vee\left(\frac{1}{|z|}f(s,y,z)\right)\wedge m\right)\wedge m, $$
$$ g_{nm}(s,y,z):=\frac{z}{|z|^2}\mathbf{1}_{\{z\neq 0\}}f_{nm}(s,y,z). $$
For all $n,m\in\mathbb{N}$, the functions $f_{nm}$ and $g_{nm}$ are
bounded. This follows from
$$ |g_{nm}(\cdot,z)|\leq\frac{1}{|z|}|f_{nm}(\cdot,z)|\leq \left|(-n)\vee\left(\frac{1}{|z|}f(\cdot,z)\right)
\wedge m\right|.$$
Furthermore
$$ f_{nm}(s,\cdot,z)\leq C(|z|^2+|z|\phi_s),\quad s\in[0,T], z\in\mathbb{R}^d. $$
The double sequence $(f_{nm})_{n,m\in\mathbb{N}}$ is pointwise
non-decreasing in $m$ for  fixed $n$, and non-increasing in $n$ for fixed
$m$. According to Lemma \ref{meassol-boundbound} and the associated corollary there
exist minimal measure solutions $\mathbb{Q}_{nm}$ of the BSDEs given by
$g_{nm}$ and $\xi$, $n,m\in\mathbb{N}$. The corresponding double sequence
$(Y^{nm})_{n,m\in\mathbb{N}}$ is non-decreasing in $m$ and non-increasing
in $n$ according to Corollary \ref{meassol-compa}. For $n\in\mathbb{N}$ define
$\widetilde{Y}^n:=\lim_{m\rightarrow\infty}Y^{nm}$. We next apply Theorem
\ref{meassol-result'} to show the existence of a.-measure solutions of the BSDE
given by $g_n:=\lim_{m\rightarrow\infty}g_{nm}$ and $\xi$. Note that $g_n$
is associated with
$$ f_{n}(\cdot,z):=(-n)\vee\left(|z|\cdot\,(-n)\vee\left(\frac{1}{|z|}f(\cdot,z)\right)\right), $$
$$ g_{n}(\cdot,z)=\frac{z}{|z|^2}\mathbf{1}_{\{z\neq 0\}}f_{n}(\cdot,z)\quad z\in\mathbb{R}^d. $$
Let us verify the hypotheses of Theorem \ref{meassol-result'}.

Condition \emph{i)}: According to Lemma \ref{meassol-BSDEBMO} applied to the BSDE
given by $f_{nm}$ and $\xi$, the control processes $Z^{nm}$ are in
$BMO(\mathbb{P})$ with a uniformly bounded BMO-norm, $n,m\in\mathbb{N}$.
Using Lemmas \ref{meassol-bmo1} and \ref{meassol-bmo2} there must exist a $p>1$ such that
\emph{i)} is satisfied.

Condition \emph{ii)} holds trivially.

Condition \emph{iii)} holds by dominated convergence and uniform
boundedness of $(Y^{nm})_{n,m\in\mathbb{N}}$.

Condition \emph{iv)} holds according to the construction of
$(g_{nm})_{n,m\in\mathbb{N}}$.

Theorem \ref{meassol-result'} provides a.-measure solutions $\mathbb{Q}_{n}$ of the
BSDE given by $g_n$ and $\xi$ and we have the associated solution pair
$(Y^n,Z^n)$, $n\in\mathbb{N}$. We have
$$ Y^{n}:=\lim_{m\rightarrow\infty}Y^{nm}=\mathbb{E}_{\mathbb{Q}_{n}}[\xi|\mathcal{F}_\cdot] $$
according to Remark \ref{meassol-resrem'2}. Furthermore according to Remark
\ref{meassol-resrem'3} for any $n\in\mathbb{N}$ $Z_n$ is in $BMO(\mathbb{P})$ with
a BMO-norm uniformly bounded in $n$, since $Z_n$ results from an a.e.-limit
of
processes with uniformly bounded BMO-norms (see Lemma \ref{meassol-bmolim}). \\
Note that $Y^{nm}\geq Y^{n'm}$ a.e. for all $m\in\mathbb{N}$ and $n<n'$
implies $Y^{n}\geq Y^{n'}$ a.e.. Thus $(Y^{n})_{n\in\mathbb{N}}$ converges
to some limit $\tilde{Y}$, and we can similarly apply Theorem \ref{meassol-result'}
to $(\mathbb{Q}_{n},g_{n},Y^n)_{n\in\mathbb{N}}$ to construct an a.-measure
solution of the BSDE given by $\tilde{g}$
and $\xi$. For its applicability we use Remark \ref{meassol-resrem'1}. \\
This means that there exists an a.-measure solution $\mathbb{Q}$ of the
BSDE given by $\tilde{g}$ and $\xi$. Using Remark \ref{meassol-resrem'3} as well as
Lemma \ref{meassol-bmolim}, we can assume that the
corresponding control process $\tilde{Z}$ is in $BMO(\mathbb{P})$. \\
Now define $\zeta:=g(\cdot,\tilde{Y},\tilde{Z})$. Then $\zeta\in
BMO(\mathbb{P})$ and hence $\mathbb{Q}:=\mathcal{E}(\zeta\bullet
W)_T\cdot\mathbb{P}$ is a probability measure. Thus it is a measure
solution of the BSDE given by $g$ and $\xi$ according to Remark \ref{meassol-alm2}.
\end{proof}

\begin{appendix}

\section{Appendix}\label{meassol-appe} We collect some well known facts and
extensions about BMO, exponential martingales and normed spaces.
The BMO prerequisites are needed for the derivation of our main existence
results for measure solutions in the preceding section, and underpin the
fact that this class of martingales plays a very important role in concepts
of control theory. The first topic we address is concerned with
integrability properties of the exponentials of BMO martingales.

\begin{lemma}\label{meassol-bmo1} For $K>0$ there exist real constants $r<0$ and $C>0$ such that
$$ \mathbb{E}[\mathcal{E}(M)^r_T]\leq C $$
for all BMO-martingales $M=Z\bullet W:=\int_0^\cdot Z_s\,\mathrm{d} W_s$ satisfying $\|Z\|_{BMO(\mathbb{P})}\leq K$.
\end{lemma}
\begin{proof}
Indeed, we may write for real $r$
$$ \mathbb{E}\left[\exp\left(rM_T-r\frac{1}{2}\langle
M\rangle_T\right)\right]= \mathbb{E}\left[\exp\left(rM_T-\langle
rM\rangle_T\right)\exp\left(\left(r^2-\frac{r}{2}\right)\langle
M\rangle_T\right)\right]\leq $$
$$ \leq \left(\mathbb{E}\left[\exp\left(2rM_T-\frac{1}{2}\langle 2rM\rangle_T\right)\right]\right)^\frac{1}{2}
\left(\mathbb{E}\left[\exp\left(r(2r-1)\langle
M\rangle_T\right)\right]\right)^\frac{1}{2}=$$
$$\hspace*{8cm}=\left(\mathbb{E}\left[\exp\left(\hat{r}\langle
M\rangle_T\right)\right]\right)^\frac{1}{2} $$
with $\hat{r}:=r(2r-1)$. Now define $r:=\frac{1}{4}-\sqrt{\frac{1}{16}+\frac{1}{4K^2}}<0$ and $C:=\sqrt{2}$. Then $\hat{r}=\frac{1}{2K^2}$. Setting $\tilde{Z}:=\sqrt{\hat{r}}Z$, we have $\|\tilde{Z}\|^2_{BMO(\mathbb{P})}\leq \frac{1}{2}$ if $\|Z\|_{BMO(\mathbb{P})}\leq K$. \\
Then according to Theorem 2.2. of \cite{meassol-kazam}
$$ \mathbb{E}\left[\exp\left(\hat{r}\int_0^T |Z_t|^2\,\mathrm{d} t\right)\right]=\mathbb{E}\left[\exp\left(\int_0^T |\tilde{Z}_t|^2\,\mathrm{d} t\right)\right]\leq\frac{1}{1-\|\tilde{Z}\|^2_{BMO(\mathbb{P})}}\leq 2=C^2. $$
\end{proof}

\begin{lemma}\label{meassol-bmo2} For $K>0$ there exist real $p>1$ and $C>0$ such that
$$ \mathbb{E}[\mathcal{E}(M)^p_T]\leq C $$
for all BMO-martingales $M=Z\bullet W$ satisfying $\|Z\|_{BMO(\mathbb{P})}\leq K$.
\end{lemma}
\begin{proof}
According to the proof of Theorem 3.1. of \cite{meassol-kazam} the inequality $\|Z\|_{BMO(\mathbb{P})}< \Phi(p)$ would imply
$$ \mathbb{E}[\mathcal{E}(M)^p_T]\leq \frac{2}{1-2(p-1)(2p-1)^{-1}\exp(p^2\|Z\|_{BMO(\mathbb{P})}(2+\|Z\|_{BMO(\mathbb{P})}))}<\infty $$
with
$$ \Phi(p):=\left(1+\frac{1}{p^2}\ln\frac{2p-1}{2(p-1)}\right)^\frac{1}{2}-1 $$
for $p>1$. \\
Note that $\|Z\|_{BMO(\mathbb{P})}< \Phi(p)$ also implies
$$ \|Z\|_{BMO(\mathbb{P})}(2+\|Z\|_{BMO(\mathbb{P})})=(\|Z\|_{BMO(\mathbb{P})}+1)^2-1<\frac{1}{p^2}\ln\frac{2p-1}{2(p-1)} $$
Since $\lim_{p\rightarrow 1}\Phi(p)=\infty$ we can choose $p>1$ such that
$K<\Phi(p)$. This proves the assertion.
\end{proof}

The following result extends Proposition 7.3 in \cite{meassol-prhedg}. It provides
conditions on the increase rate of the drift term of a stochastic equation
which give bounds on the BMO norm of its stochastic integral part.

\begin{lemma}\label{meassol-BSDEBMO} Let $Y$, $Z$, $X$, $\psi$, $\varphi$ be some progressive processes such that
$Y$ is bounded and
$$ Y_t=Y_T+\int_t^T X_s\,\mathrm{d} s-\int_t^T Z_s\,\mathrm{d} W_s \quad \textrm{a.s. }\forall t\in[0,T]. $$
Let $X\leq \psi^2+|Z|\varphi+C|Z|^2$ with some $C>0$ and $\varphi,\psi\in BMO(\mathbb{P})$.

Then there exists a constant $K$ which only depends on
$\|\varphi\|_{BMO(\mathbb{P})}\vee\|\psi\|_{BMO(\mathbb{P})}\vee C
\vee\|Y\|_\infty$ such that $\|Z\|_{BMO(\mathbb{P})}\leq K$.
\end{lemma}
\begin{proof}
By hypothesis
$$ X\leq\psi^2+|Z|\varphi+C|Z|^2\leq
(\psi^2+\frac{1}{2}\varphi^2)+(C+\frac{1}{2})|Z|^2. $$ Define
$\tilde{\psi}:=\sqrt{\psi^2+\frac{1}{2}\varphi^2}\in BMO(\mathbb{P})$ and
$\tilde{C}:=C+\frac{1}{2}$.

For $0\le t\le T$ we can write
$$ Y_t=Y_0-\int_0^t X_s\,\mathrm{d} s+\int_0^t Z_s\,\mathrm{d} W_s. $$
Let $\beta\in\mathbb{R}$. Using It\^{o}'s formula we have
$$ \exp(\beta Y_t)=\exp(\beta Y_0)-\int_0^t \beta\exp(\beta Y_s)X_s\,\mathrm{d} s+\hspace*{4cm}$$
$$\hspace*{4cm}+\int_0^t \beta\exp(\beta Y_s)Z_s\,\mathrm{d} W_s+
\frac{\beta^2}{2}\int_0^t \exp(\beta Y_s)|Z_s|^2\,\mathrm{d} s $$
or
$$ \exp(\beta Y_t)=\exp(\beta Y_T)+\int_t^T \beta\exp(\beta Y_s)X_s\,\mathrm{d} s-\hspace*{4cm}$$
$$\hspace*{4cm}-\int_t^T \beta\exp(\beta Y_s)Z_s\,\mathrm{d} W_s-
\frac{\beta^2}{2}\int_t^T \exp(\beta Y_s)|Z_s|^2\,\mathrm{d} s. $$ We obtain
$$ \beta\int_t^T \exp(\beta Y_s)\left(\frac{\beta}{2}|Z_s|^2-X_s\right)\,\mathrm{d} s=\exp(\beta Y_T)-
\exp(\beta Y_t)-\int_t^T \beta\exp(\beta Y_s)Z_s\,\mathrm{d} W_s, $$ and thus by
hypothesis
$$ \beta\int_t^T \exp(\beta Y_s)\left(\frac{\beta}{2}|Z_s|^2+\tilde{\psi}^2_s-X_s\right)\,\mathrm{d} s=$$
$$=\exp(\beta Y_T)-\exp(\beta Y_t)+\beta\int_t^T \exp(\beta Y_s)\tilde{\psi}^2_s\,\mathrm{d} s-\int_t^T
\beta\exp(\beta Y_s)Z_s\,\mathrm{d} W_s. $$ Setting $\beta:=2\tilde{C}+2=2C+3$ we
have
$$ |Z|^2\leq\frac{\beta}{2}|Z|^2+\tilde{\psi}^2-X. $$
Applying conditional expectations we get for $0\le t\le T$
$$ \mathbb{E}\left[\beta\int_t^T \exp(\beta Y_s)|Z_s|^2\,\mathrm{d} s\bigg|\mathcal{F}_t\right]\leq $$
$$\leq\mathbb{E}\left[\exp(\beta Y_T)-\exp(\beta Y_t)+\beta\int_t^T \exp(\beta Y_s)
(\psi^2+\frac{1}{2}\varphi^2)\,\mathrm{d} s\bigg|\mathcal{F}_t\right]. $$ Since
$\exp(\beta Y)$ is bounded from below and from above and since
$\psi,\varphi\in BMO(\mathbb{P})$ this shows 
$$\|Z\|_{BMO(\mathbb{P})}\leq
K<\infty,$$ with some $K$ which can be already determined by only knowing an
upper bound for $\beta$, $\|Y\|_\infty$, $\|\psi\|_{BMO(\mathbb{P})}$ and
$\|\varphi\|_{BMO(\mathbb{P})}$ and a lower bound for $\beta$ (which is
$3$).
\end{proof}

The following is a Fatou type property of BMO norms.

\begin{lemma}\label{meassol-bmolim}
Let $(Z^n)$ be a sequence of progressive processes converging to some
progressive $Z$ in measure (w.r.t. $\mathbb{P}\otimes\lambda_{[0,T]}$). \\
Then
$\|Z\|_{BMO(\mathbb{P})}\leq
\liminf_{n\to\infty}\|Z^n\|_{BMO(\mathbb{P})}$.
\end{lemma}
\begin{proof} Assume the contrary. Then there must exist a subsequence 
$(Z^{n_k})$ s.t. $$\|Z\|_{BMO(\mathbb{P})}>
\liminf_{k\to\infty}\|Z^{n_k}\|_{BMO(\mathbb{P})}$$ and $Z^{n_k}\rightarrow
Z$ a.e. for $k\to\infty$. However by Fatou's inequality we have for $0\le
t\le T$
$$ \mathbb{E}\left[\int_t^T|Z_s|^2\,\mathrm{d} s\bigg|\mathcal{F}_t\right] \leq\liminf_{k\to\infty}\mathbb{E}\left[\int_t^T|Z^{n_k}_s|^2\,\mathrm{d} s\bigg|\mathcal{F}_t\right]
\leq \liminf_{k\to\infty}\|Z^{n_k}\|_{BMO(\mathbb{P})} $$ almost surely.
This implies
$\|Z\|_{BMO(\mathbb{P})}\leq\liminf_{k\to\infty}\|Z^{n_k}\|_{BMO(\mathbb{P})}$,
yielding a contradiction.
\end{proof}

Next is the well-known representation property of exponential martingales on a Brownian basis.

\begin{lemma}\label{meassol-density}
Let $\mathbb{Q}\sim\mathbb{P}$ be a probability measure. Define
$R_T:=\frac{\mathrm{d}\mathbb{Q}}{\mathrm{d}\mathbb{P}}$ as the Radon-Nikodym
derivative. Then the martingale $R:=\mathbb{E}[R_T|\mathcal{F}_\cdot]$ can be
written as
$$ R=\exp\left(\int_0^\cdot\zeta_s\,\mathrm{d}W_s-\frac{1}{2}\int_0^\cdot|\zeta_s|^2\,\mathrm{d}s\right) $$
with some progressively measurable process $\zeta$ such that $\int_0^T|\zeta_s|^2\,\mathrm{d}s<\infty$ a.s.
\end{lemma}
\begin{proof} This is a consequence of Proposition 1.6 of Chapter VIII in \cite{meassol-revuzyor} and the martingale
representation theorem.
\end{proof}

We next recall a well known dual characterization of H\"older norms.

\begin{lemma}\label{meassol-lpnorm} Let $\mathbb{Q}$ be a probability measure on $(\Omega,\mathcal{F})$. Let $Y$
be non-negative and $\mathcal{F}$-measurable. Let $p,q>1$ such that
$\frac{1}{p}+\frac{1}{q}=1$. Denote $\mathcal{X}$ the set of all bounded
$\mathcal{F}$-measurable $X>0$. Then
$$ \mathbb{E}_\mathbb{Q}[Y^p]^\frac{1}{p}=\sup_{X\in\mathcal{X}}
\frac{\mathbb{E}_\mathbb{Q}[Y\cdot
X]}{\mathbb{E}_\mathbb{Q}[X^q]^\frac{1}{q}}. $$
\end{lemma}
\begin{proof}
First use H\"older's inequality to get
$$ \frac{\mathbb{E}_\mathbb{Q}[Y\cdot X]}{\mathbb{E}_\mathbb{Q}[X^q]^\frac{1}{q}}\leq
\frac{(\mathbb{E}_\mathbb{Q}[Y^p])^\frac{1}{p}
(\mathbb{E}_\mathbb{Q}[X^q])^\frac{1}{q}}{\mathbb{E}_\mathbb{Q}[X^q]^\frac{1}{q}}.
$$ This implies
$$ \sup_{X\in\mathcal{X}}
\frac{\mathbb{E}_\mathbb{Q}[Y\cdot
X]}{\mathbb{E}_\mathbb{Q}[X^q]^\frac{1}{q}}\leq
\mathbb{E}_\mathbb{Q}[Y^p]^\frac{1}{p}. $$ Secondly, defining for
$n,m\in\mathbb{N}$ the bounded and positive $X_{nm}:=(\frac{1}{m}+Y\wedge
n)^{p-1}$, and assuming $0<\mathbb{E}_\mathbb{Q}[Y]<\infty$ for a moment,
we may write
$$ \sup_{X\in\mathcal{X}}\frac{\mathbb{E}_\mathbb{Q}[Y\cdot X]}{(\mathbb{E}_\mathbb{Q}[X^q])^\frac{1}{q}}\geq
\sup_{n,m}\frac{\mathbb{E}_\mathbb{Q}[Y\cdot X_{nm}]}{(\mathbb{E}_\mathbb{Q}[X_{nm}^q])^\frac{1}{q}}\geq
\lim_{n\rightarrow\infty}\lim_{m\rightarrow\infty}\frac{\mathbb{E}_\mathbb{Q}[(Y\wedge n)\cdot X_{nm}]}{(\mathbb{E}_\mathbb{Q}[X_{nm}^q])^\frac{1}{q}}= $$
$$ =\lim_{n\rightarrow\infty}\frac{\mathbb{E}_\mathbb{Q}[(Y\wedge n)\cdot (Y\wedge n)^{p-1}]}{(\mathbb{E}_\mathbb{Q}[(Y\wedge n)^{(p-1)q}])^\frac{1}{q}}=
\lim_{n\rightarrow\infty}(\mathbb{E}_\mathbb{Q}[(Y\wedge
n)^p])^\frac{1}{p}= (\mathbb{E}_\mathbb{Q}[Y^p])^\frac{1}{p}. $$
Here we use $\mathbb{E}_\mathbb{Q}[(Y\wedge n)^p]>0$ for $n$ large enough. \\
Now if $\mathbb{E}_\mathbb{Q}[Y]=0$ the claim is trivial, since this implies $Y=0$ $\mathbb{Q}$-a.s. \\
If $\mathbb{E}_\mathbb{Q}[Y]=\infty$ the claim becomes trivial as well,
since then $\mathbb{E}_\mathbb{Q}[Y^p]=\infty$ and we can set $X=1$.
\end{proof}

We finally show that the regularization of a continuous function $f$ by
(\ref{meassol-infconvolution}) indeed yields a sequence of uniformly continuous
functions converging monotonously to $f$. Denote by $\overline{A}$ the
closure of a set $A$ in a topological space.

\begin{lemma}\label{meassol-uniformcontinuity}
Let $R, M>0$, and let $f:[-R,R]\times \mathbb{R}^m \to
\mathbb{R}$ be bounded by $M$ and continuous. Define
$f_n:[-R,R]\times \mathbb{R}^m \to \mathbb{R}$ by
$$f_n(y,z) = \inf_{\hat z\in\mathbb{R}^m, \hat{y}\in [-R,R]}
\left\{ f(\hat{y}, \hat{z}) + \frac{n}{0\vee (|z|-n)\wedge 1}
\left|{y\choose z}-{\hat{y}\choose\hat{z}}\right|\right\},\hspace*{1cm}$$
$$\hspace*{8cm} z\in\mathbb{R}^d, y\in\mathbb{R}, n\in\mathbb{N},$$
with the conventions of (\ref{meassol-infconvolution}). Then we have
\begin{description}
\item[](i) $f_n$ is bounded by $M$ for any $n\in\mathbb{N}$,
\item[](ii) $f_n = f$ on $[-R,R]\times \overline{B_n(0)}$,
\item[](iii) $f_n$ is uniformly continuous for any $n\in\mathbb{N}$,
\item[](iv) $f_n\le f_{n+1}\le f$ for any $n\in\mathbb{N}$.
\end{description}
\end{lemma}
\begin{proof}
(i): This follows from the inequality
$$-M\le \inf_{\hat z\in\mathbb{R}^m, \hat{y}\in \overline{B_R(0)}}
\left\{ -M + \frac{n}{0\vee (|z|-n)\wedge 1} \left|{y\choose
z}-{\hat{y}\choose\hat{z}}\right|\right\} \le \hspace*{2cm}$$
$$\hspace*{7cm} \leq f_n(y,z) \le f(y,z)\le M,$$
valid for $y\in\mathbb{R}, z\in\mathbb{R}^m.$

(ii): Note that for $|z|\le n$ we have $\frac{n}{0\vee (|z|-n)\wedge 1} =
\infty$, and therefore for $|z|\le n, y\in\overline{B_R(0)}$
$$ f_n(y,z) = \inf_{\hat z\in\mathbb{R}^m, \hat{y}\in \overline{B_R(0)}}
\left\{ f(\hat{y}, \hat{z}) + \frac{n}{0\vee (|z|-n)\wedge 1}
\left|{y\choose z}-{\hat{y}\choose\hat{z}}\right|\right\} = f(y,z).$$

(iv): For $z\in\mathbb{R}^m, n\in\mathbb{N}$ we note $|z|-n\ge |z|-(n+1),$
and hence $$\frac{n}{0\vee (|z|-n)\wedge 1}\le \frac{n+1}{0\vee
(|z|-(n+1))\wedge 1}\le \infty.$$ This implies $f_n\le f_{n+1}\le f$.

(iii): For $n\in\mathbb{N}, t\ge 0$ denote $\phi_n(t) = \frac{n}{0\vee
(t-n)\wedge 1}$. Then
$$\phi_n(t) = \left\{
\begin{array}{ll}
\infty, & t\le n,\\
\frac{n}{t-n}, & n<t\le n+1,\\
n, & t>n+1,
\end{array}
\right.$$ and hence for any $\epsilon>0$ $\phi_n$ is bounded and Lipschitz
continuous on $[n+\epsilon, \infty).$ Hence for any $\epsilon>0$ and
$\hat{z}\in\mathbb{R}^m, \hat{y}\in\overline{B_R(0)}$ the mapping
$$A_{\hat{z},\hat{y}}: \overline{B_R(0)}\times(\mathbb{R}^m\setminus
B_{n+\epsilon}(0))\ni (y,z)\longmapsto \phi_n(|z|)\left|{y\choose
z}-{\hat{y}\choose\hat{z}}\right|$$ is Lipschitz continuous with Lipschitz
constant $L_\epsilon$ independent of $\hat{y}\in\overline{B_R(0)},
\hat{z}\in\mathbb{R}^m$. To abbreviate, let $M_\epsilon =
\overline{B_R(0)}\times(\mathbb{R}^m\setminus B_{n+\epsilon}(0))$. Then we
can write for $(y,z), (y', z')\in M_\epsilon, n\in\mathbb{N}$
$$f(\hat{y},\hat{z}) + \phi_n(|z|)\left|{y\choose
z}-{\hat{y}\choose\hat{z}}\right| \le f(\hat{y},\hat{z}) +
\phi_n(|z'|)\left|{y'\choose z'}-{\hat{y}\choose\hat{z}}\right| +
L_\epsilon \left|{y\choose z}-{y'\choose z'}\right|,$$ from which we deduce
by eventually switching the roles of $(y,z)$ and $(y',z')$
\begin{eqnarray*}
f_n(y,z) &\le& f_n(y',z') + L_\epsilon \left|{y\choose z}-{y'\choose
z'}\right|,\\
f_n(y',z') &\le& f_n(y,z) + L_\epsilon \left|{y\choose z}-{y'\choose
z'}\right|,
\end{eqnarray*}
and therefore
$$|f_n(y,z) - f_n(y',z')| \le L_\epsilon \left|{y\choose z}-{y'\choose
z'}\right|.$$ As a conclusion we obtain that for $\epsilon>0$ $f_n$ is
Lipschitz continuous on $M_\epsilon$, and therefore uniformly continuous.
As $(\overline{B_R(0)}\times \mathbb{R}^m) \setminus M_\epsilon$ is bounded,
and since continuous functions on compact sets are uniformly continuous, to
show (iii), it is therefore enough to prove that $f_n$ is continuous. By
construction this is clear for $(y,z)$ such that either $|z|<n$ or $|z|>n.$
Let therefore $(\tilde{y},
\tilde{z})\in [-R,R]\times\mathbb{R}^m$ such that $|\tilde{z}|=n$
be given. We have for $y\in [-R,R], z\in\mathbb{R}^m$
$$f_n(y,z) = \inf_{\hat{y}\in \overline{B_R(0)}, {\hat{y}\choose \hat{z}}\in
\overline{B_{\frac{2M}{\phi_n(|z|)}} ({y\choose z})}} \left\{ f(\hat{y},
\hat{z}) + \phi_n(|z|)\left|{y\choose
z}-{\hat{y}\choose\hat{z}}\right|\right\},$$ hence
$$\inf_{\hat{y}\in \overline{B_R(0)}, {\hat{y}\choose \hat{z}}\in \overline{B_{\frac{2M}{\phi_n(|z|)}}
({y\choose z})}} \left\{ f(\hat{y}, \hat{z})\right\} \le f_n(y,z),$$ and
therefore by $\phi_n(|z|)\to \infty$ for $z\to\tilde{z}$ and (ii)
\begin{eqnarray*}
\liminf_{(y,z)\to (\tilde{y},\tilde{z})} f_n(y,z) &\ge& \liminf_{(y,z)\to
(\tilde{y},\tilde{z})} \inf_{\hat{y}\in \overline{B_R(0)}, {\hat{y}\choose
\hat{z}}\in \overline{B_{\frac{2M}{\phi_n(|z|)}}
({y\choose z})}} \left\{ f(\hat{y},
\hat{z})\right\} = f(\tilde{y},\tilde{z}) = f_n(\tilde{y},\tilde{z}),\\
\limsup_{(y,z)\to (\tilde{y},\tilde{z})} f_n(y,z) &\le& \limsup_{(y,z)\to
(\tilde{y},\tilde{z})} f(y,z)= f(\tilde{y},\tilde{z}) = f_n(\tilde{y},\tilde{z}).
\end{eqnarray*}
This obviously implies the continuity of $f_n$ in $(\tilde{y},\tilde{z})$,
and concludes the proof of (iii). \\
\end{proof}

\end{appendix}

\end{document}